\documentclass[hidelinks,onefignum,onetabnum]{siamart251216}

\usepackage{mathtools}
\usepackage{amsfonts,amssymb}
\usepackage{booktabs}
\usepackage{array}
\usepackage{xcolor}
\usepackage{microtype}
\usepackage{enumitem}
\usepackage[numbers,sort&compress]{natbib}
\usepackage{algorithm}
\usepackage{algpseudocode}

\ifpdf
  \DeclareGraphicsExtensions{.pdf,.png,.jpg,.jpeg}
\fi

\newsiamremark{remark}{Remark}
\newsiamremark{assumption}{Assumption}
\crefname{assumption}{Assumption}{Assumptions}
\Crefname{assumption}{Assumption}{Assumptions}

\newcommand{\RLleft}[2]{{}_{-\infty}D^{#1}_{#2}}
\newcommand{\RLright}[2]{{}_{#2}D^{#1}_{\infty}}
\newcommand{\Cleft}[2]{{}^{C}_{-\infty}D^{#1}_{#2}}
\newcommand{\Lcgmy}{\mathcal{L}_{\mathrm{CGMY}}}
\newcommand{\bsphi}{\phi}

\newcommand{\Aminus}{A^{-}}
\newcommand{\Aplus}{A^{+}}
\newcommand{\norm}[1]{\left\|#1\right\|}

\newcommand{\vp}{\ensuremath{\mathrm{vp}}}

\headers{B-Spline Galerkin CGMY Pricing}{D. Damircheli}

\title{A B-Spline Galerkin Method for Option Pricing under the CGMY Tempered Stable Process}

\author{Davood Damircheli\thanks{Department of Mathematics, Louisiana State University, Baton Rouge, LA 70803, USA.}}

\ifpdf
\hypersetup{
  pdftitle={A B-Spline Galerkin Method for Option Pricing under the CGMY Tempered Stable Process},
  pdfauthor={Davood Damircheli}
}
\fi

\begin{document}

\maketitle

\begin{abstract}
We develop and analyze a B-spline scaling-function Galerkin method for the
tempered fractional partial differential equation that governs European and
American option prices under the CGMY tempered stable L\'evy process. The
spatial generator couples exponentially tempered left- and right-sided
Riemann--Liouville derivatives of order $Y\in(0,2)$, whose nonlocality and
asymmetry make accurate deterministic discretization challenging. We derive a
corrected Fourier-symbol construction of the tempered connection-coefficient
matrices whose interior blocks are exactly Toeplitz, carry the correct
fractional scaling, and admit $O(N\log N)$ matrix--vector products by the
fast Fourier transform. In the infinite-variation regime $Y\in(1,2)$ we prove
full-line and finite-domain G{\aa}rding inequalities --- coercive whenever
the risk-free rate is positive --- by a zero-extension reduction to the
full-line symbol, and on this foundation establish elliptic-projection,
semidiscrete, and fully discrete error estimates with spatial rate
$O(h^{p-Y/2})$, together with an unconditional shifted-energy stability
estimate for the forced Crank--Nicolson scheme with mesh-independent
constants; for positive rates the admissible shift vanishes, so the theory
governs the plain scheme exactly as implemented. A single backward solve
prices an entire strike strip at fixed maturity with smooth, analytically
differentiated Greeks, and the same discretization prices American puts
through a projected linear complementarity solve, with early-exercise
premiums and exercise boundaries validated against a resolved deterministic
reference and a Longstaff--Schwartz Monte Carlo cross-check. Calibration to
S\&P~500 option surfaces across low-volatility, crisis, and recovery regimes
attains implied-volatility errors near two vol-points on the well-identified
surfaces, and an accuracy-matched comparison with a stabilized
Gr\"unwald--Letnikov baseline delineates precisely when the high-order
variational structure pays off.
\end{abstract}

\begin{keywords}
fractional PDE, CGMY process, tempered stable distribution,
B-spline scaling functions, Galerkin method, connection coefficients,
Riemann--Liouville derivative, option pricing, model calibration
\end{keywords}

\begin{MSCcodes}
65M60, 65M12, 65T60, 91G20, 60G51, 26A33
\end{MSCcodes}

\section{Introduction}
\label{sec:intro}

Exponential L\'{e}vy models correct the well-documented failure of the
Gaussian Black--Scholes framework to reproduce volatility smiles, skewness, and
heavy tails in observed option data \citep{BlackScholes1973,ContTankov2004}.
Among these models the Carr--Geman--Madan--Yor (CGMY) tempered stable process
is distinguished by its ability to represent both the finite arrival of large
jumps and the infinite activity of small ones through a single fine-structure
index $Y$ \citep{CarrGemanMadanYor2002}. Under CGMY dynamics the European and
American pricing problems become nonlocal tempered fractional PDEs: in the
log-price variable the spatial generator couples left- and right-sided
exponentially tempered Riemann--Liouville derivatives of order $Y$. These
operators are nonlocal and nonsymmetric, and standard discretizations produce
dense matrices whose assembly and application dominate the cost of the repeated
solves required for implied-volatility calibration and for the smooth Greeks
demanded by risk management.

Existing deterministic approaches each address part of this difficulty.
Finite-difference and PIDE schemes must treat the singular small-jump behaviour
of the L\'{e}vy measure and typically lead to dense or carefully compensated
operators \citep{ContVoltchkova2005}; a naive uncompensated implementation is
inaccurate or unstable in the infinite-variation regime, whereas a stabilized
implementation is competitive at moderate accuracy. Fourier and Fourier-cosine
(COS) methods are extremely efficient for European pricing when a characteristic
function is available \citep{FangOosterlee2008,Oosterlee2019}, but their global
transforms are less naturally aligned with localized bases, boundary forcing,
and early-exercise constraints. Wavelet and multiresolution methods provide
Riesz-basis preconditioning and compression \citep{Cohen2003}, yet the exact
B-spline connection-coefficient structure for \emph{exponentially tempered}
fractional operators, and the resulting Toeplitz interior blocks, have not been
exploited; many constructions instead require singular-integral
regularizations. Monte Carlo remains the universal fallback for path-dependent
CGMY payoffs, but exact increment simulation is difficult
\citep{MadanSeneta1990,AsmussenRosinski2001,GaoFu2011}, and its
$\mathcal{O}(1/\sqrt{M})$ probabilistic error and point-by-point evaluation make
it ill-suited to producing smooth surfaces and stable Greeks across a full
calibration grid. These considerations motivate a deterministic, high-order,
structure-preserving surface solver.

We propose a B-spline scaling-function Galerkin method for the CGMY backward
FPDE. Using translates of compactly supported cardinal B-splines at a fixed
resolution level, we derive exact tempered connection coefficients whose
interior matrices are Toeplitz, so that matrix--vector products cost
$\mathcal{O}(N\log N)$ by FFT rather than $\mathcal{O}(N^2)$. We develop the
stability and convergence theory in the infinite-variation regime
$Y\in(1,2)$, where the jump singularity is most pronounced, and we exploit the
fact that a single backward solve prices an entire strike range at a fixed
maturity, providing a foundation for risk-neutral calibration directly under
$\mathbb{Q}$ \citep{HainautColwell2014}.

The main contributions of this paper are threefold.

\emph{Operator structure.} We derive corrected Fourier-symbol representations
for the connection-coefficient matrices $\Aminus$ and $\Aplus$ of the
exponentially conjugated left- and right-sided tempered Riemann--Liouville
operators. The full-line interior matrices are Toeplitz, carry the correct
fractional scaling factor $2^{jY}$, and admit FFT-based matrix--vector products
in $O(N\log N)$ operations; closed-form positive-part formulae in the
untempered limit serve as implementation checks.

\emph{Analysis.} For the infinite-variation regime $Y\in(1,2)$ we prove
full-line \emph{and} finite-domain G{\aa}rding inequalities; the finite-domain
inequality is obtained by a zero-extension reduction to the full-line Fourier
symbol, so that no unverified coercivity assumption remains.  A monotonicity
property of the tempered symbol shows that both inequalities are in fact
coercive ($c_2=0$) whenever the risk-free rate $r$ is positive, which is the
case in every experiment reported here.  On this foundation we prove a shifted
elliptic projection estimate, a semidiscrete parabolic error estimate, and a
fully discrete convergence bound with spatial rate $O(h^{p-Y/2})$ for B-spline
order $p$. We also establish a forced, shifted-energy Crank--Nicolson
stability estimate with constants independent of $h$ and $\Delta\tau$,
controlling the growth of the option variable in the presence of boundary
forcing; for $r>0$ the admissible shift is $\sigma=0$, so the estimate governs
the plain Crank--Nicolson iteration exactly as implemented.

\emph{Validation and applications.} We verify the predicted orders on
manufactured and European benchmarks, extend the same Galerkin operator to
American puts through a projected-SOR linear complementarity solve, and
calibrate the model to real S\&P~500 surfaces across three market regimes. A
work--precision comparison and a crisis-date calibration benchmark against a
stabilized Gr\"unwald--Letnikov solver show that the first-order baseline is
faster at loose price tolerances and in the reported Gamma sweep, whereas the
B-spline method is the more reliable high-order variational solver when tight
pricing accuracy, exact spline-based Greek recovery, or compatibility with the
American obstacle problem is required.

Section~\ref{sec:mathsetup} derives the CGMY backward FPDE, introduces the
weighted fractional Sobolev space, and establishes the variational
formulation.
Section~\ref{sec:bspline_wavelets} reviews B-spline scaling functions and
derives the connection coefficient matrices.
Section~\ref{sec:galerkin_scheme} describes the B-spline Galerkin spatial
discretisation and Crank--Nicolson time-stepping.
Section~\ref{sec:analysis} proves convergence and stability.
Section~\ref{sec:numerics} presents numerical experiments including
manufactured-solution convergence verification, comparison with
GL and COS benchmarks, implied-volatility and Greek calculations, and the
American put extension.
Section~\ref{sec:calibration} describes the calibration methodology and
reports results on S\&P 500 data.
Section~\ref{sec:conclusion} concludes and outlines extensions.
Appendix~\ref{app:operators} collects notation for all fractional operators
used in the paper.  Complete proofs of the semidiscrete, stability, and fully
discrete theorems, auxiliary proofs for the connection coefficients, and the
full calibration pipeline are collected in the accompanying supplementary
materials.
\section{Mathematical Setup}
\label{sec:mathsetup}

\subsection{The CGMY Backward FPDE and Risk-Neutral Parameters}
\label{ssec:cgmy_fpde}

Let $S_t=e^{X_t}$ be the asset price under a risk-neutral measure
$\mathbb{Q}$, and let $V(S,t)$ be the price of a European payoff
$\psi(S_T)$ at maturity $T$.  Throughout the paper we use the
\emph{log-price} variable
\begin{equation}
\label{eq:log_price_convention}
x=\log S,
\end{equation}
not log-moneyness.  This convention is consistent with the put payoff
$\psi(x)=(K-e^x)^+$ and with the Greeks formulae in
Section~\ref{ssec:greeks}.

The CGMY parameters are assumed to satisfy
\begin{equation}
\label{eq:cgmy_parameter_assumption}
C>0,\qquad G>0,\qquad M>1,\qquad Y\in(0,2),\qquad Y\ne 1 .
\end{equation}
The condition $M>1$ is the risk-neutral exponential-moment condition needed
for the stock-price martingale correction below to be real-valued.  Indeed,
for noninteger $Y$, the term $(M-1)^Y$ is real only under this condition.  The
martingale correction is
\begin{equation}
\label{eq:omega}
\omega_S
= C\Gamma(-Y)\bigl[(M-1)^Y - M^Y + (G+1)^Y - G^Y\bigr].
\end{equation}
Define also
\begin{equation}
\label{eq:lambda0_def}
\lambda_0 := r + C\Gamma(-Y)(M^Y+G^Y).
\end{equation}

Following Cartea and del Castillo-Negrete~\cite{CarteadelCastilloNegrete2007}, the log-price option value
$V=V(x,t)$ satisfies the backward CGMY FPDE
\begin{align}
\label{eq:cgmy_fpde}
\frac{\partial V}{\partial t}
+ (r - \omega_S)\frac{\partial V}{\partial x}
&+ C\Gamma(-Y)\,e^{Mx}\,\RLright{Y}{x}\!\bigl(e^{-Mx}V\bigr) \notag\\
&+ C\Gamma(-Y)\,e^{-Gx}\,\RLleft{Y}{x}\!\bigl(e^{Gx}V\bigr)
= \lambda_0 V,
\end{align}
with terminal condition $V(x,T)=\psi(x)$.  We introduce backward time
$\tau=T-t$ and write
\begin{equation}
\label{eq:initial}
V(x,0)=\psi(x),\qquad
\psi(x)=\bigl(K-e^x\bigr)^+\quad\text{for a European put.}
\end{equation}

The two tempered Riemann--Liouville operators appearing in the generator are
\begin{equation}
\label{eq:tempered_ops_in_setup}
\mathcal D_G^-u=e^{-Gx}\,{}_{-\infty}D_x^Y(e^{Gx}u),
\qquad
\mathcal D_M^+u=e^{Mx}\,{}_{x}D_\infty^Y(e^{-Mx}u).
\end{equation}
With these definitions, the backward-time equation is
\begin{equation}
\label{eq:backward_time_strong}
\partial_\tau V+\mathcal A V=0,
\qquad V(x,0)=\psi(x),
\end{equation}
where the single operator used throughout the theory and the discretisation is
\begin{equation}
\label{eq:backward_operator_A}
{
\mathcal A u
:=-(r-\omega_S)u_x
-C\Gamma(-Y)\mathcal D_M^+u
-C\Gamma(-Y)\mathcal D_G^-u
+\lambda_0 u .}
\end{equation}
All Galerkin matrices in Section~\ref{sec:galerkin_scheme} are assembled from
this same operator.  This removes the common ambiguity between the
forward-time generator and the backward-time pricing operator.

\begin{remark}[Activity and variation of the CGMY process]
\label{rem:activity_variation}
Near the origin, the CGMY L\'evy density behaves like
$|y|^{-1-Y}\,dy$.  Hence
$\int_{|y|<1}\nu(dy)=\infty$ for every $Y>0$, so the process has infinite
activity throughout the range considered here.  The distinction is instead
variation: for $Y\in(0,1)$,
$\int_{|y|<1}|y|\nu(dy)<\infty$ and the process has finite variation, whereas
for $Y\in(1,2)$ it has infinite variation.  The B-spline connection
coefficient construction applies to both regimes, but the rigorous
variational stability and convergence analysis below is restricted to the
infinite-variation case $Y\in(1,2)$.
\end{remark}

The problem is solved numerically on a truncated interval
\begin{equation}
\label{eq:omega_interval}
\Omega=(x_{\min},x_{\max}),
\qquad x_{\min}\ll\log K\ll x_{\max},
\end{equation}
with exterior data described in Section~\ref{ssec:bc}.

\subsection{Full-Line Symbol, Variational Form, and G\texorpdfstring{\aa}{aa}rding Inequality}
\label{ssec:wellposedness}

The fractional part of \eqref{eq:backward_operator_A} is nonselfadjoint, and
therefore the variational form should not be replaced by same-sided fractional
``squares.''  We define the form directly from the backward-time operator.
On the full line, the Fourier symbols of the tempered operators are
\begin{equation}
\label{eq:tempered_symbols_setup}
\widehat{\mathcal D_G^-u}(\xi)=(G+i\xi)^Y\widehat u(\xi),
\qquad
\widehat{\mathcal D_M^+u}(\xi)=(M-i\xi)^Y\widehat u(\xi),
\end{equation}
with the principal branch of the complex power; see
Proposition~\ref{prop:tempered_symbols}.  Hence the symbol of $\mathcal A$ is
\begin{equation}
\label{eq:A_symbol}
A(\xi)
=-(r-\omega_S)i\xi
-C\Gamma(-Y)\Bigl[(M-i\xi)^Y+(G+i\xi)^Y\Bigr]
+\lambda_0 .
\end{equation}

For $s=Y/2$ define the full-line energy space by the Fourier norm
\begin{equation}
\label{eq:HMG_fourier_norm}
\|u\|_{H^s_{M,G}(\mathbb R)}^2
:=\frac{1}{2\pi}\int_{\mathbb R}
\Bigl(1+|M-i\xi|^{2s}+|G+i\xi|^{2s}\Bigr)|\widehat u(\xi)|^2\,d\xi .
\end{equation}
For fixed $M,G>0$, this norm is equivalent to the usual $H^s(\mathbb R)$
norm.  The full-line bilinear form is
\begin{equation}
\label{eq:bilinear}
a_{\mathbb R}(u,w):=\langle \mathcal A u,w\rangle
=\frac{1}{2\pi}\int_{\mathbb R}A(\xi)\widehat u(\xi)
\overline{\widehat w(\xi)}\,d\xi .
\end{equation}
Equivalently, for smooth functions,
\begin{align}
\label{eq:physical_bilinear_duality}
a_{\mathbb R}(u,w)
&=-(r-\omega_S)\langle u_x,w\rangle_{H^{-s},H^s}
-C\Gamma(-Y)\langle\mathcal D_M^+u,w\rangle_{H^{-s},H^s} \notag\\
&\quad
-C\Gamma(-Y)\langle\mathcal D_G^-u,w\rangle_{H^{-s},H^s}
+\lambda_0(u,w)_{L^2(\mathbb R)} .
\end{align}
The first-order drift is intentionally written as a duality pairing.  Since
$s>1/2$ in the theory below, Plancherel's theorem gives
\begin{equation}
\label{eq:drift_duality_bound}
|\langle u_x,w\rangle_{H^{-s},H^s}|
\le C_s\|u\|_{H^s(\mathbb R)}\|w\|_{H^s(\mathbb R)} .
\end{equation}

The coercivity statement below rests on a monotonicity property of the
tempered symbol that appears not to have been exploited in this context.

\begin{lemma}[Monotone decay of the tempered symbol's real part]
\label{lem:symbol_monotone}
Let $a>0$ and $Y\in(1,2)$.  The function
$\xi\mapsto\operatorname{Re}(a+i\xi)^Y$ is even and strictly decreasing in
$|\xi|$.  In particular,
\begin{equation}
\label{eq:symbol_monotone}
\operatorname{Re}(a\pm i\xi)^Y\le a^Y
\qquad\text{for all }\xi\in\mathbb R,
\end{equation}
with equality only at $\xi=0$.
\end{lemma}

The proof --- differentiation of $\operatorname{Re}(a+i\xi)^Y$ on the
principal branch --- is given in the supplementary materials
(section~SM1).

\begin{theorem}[Full-line G\texorpdfstring{\aa}{aa}rding inequality, coercivity, and shifted well-posedness]
\label{thm:wellposed}
Let $Y\in(1,2)$, $s=Y/2$, and assume
\eqref{eq:cgmy_parameter_assumption}.  Then there exist constants
$C_{\rm cont}>0$, $c_1>0$, and $c_2\ge0$, depending on the CGMY parameters and
$r$ but not on $u$ or $w$, such that
\begin{equation}
\label{eq:full_line_continuity}
|a_{\mathbb R}(u,w)|
\le C_{\rm cont}\|u\|_{H^s_{M,G}(\mathbb R)}
\|w\|_{H^s_{M,G}(\mathbb R)},
\end{equation}
and
\begin{equation}
\label{eq:garding}
\operatorname{Re}a_{\mathbb R}(u,u)
\ge c_1\|u\|_{H^s_{M,G}(\mathbb R)}^2
-c_2\|u\|_{L^2(\mathbb R)}^2 .
\end{equation}
Consequently, for
\begin{equation}
\label{eq:shifted_bilinear}
a_{\mathbb R,\sigma}(u,w)
:=a_{\mathbb R}(u,w)+\sigma(u,w)_{L^2(\mathbb R)},
\end{equation}
any $\sigma>c_2$ yields the shifted coercivity estimate
\begin{equation}
\label{eq:shifted_coercivity}
\operatorname{Re}a_{\mathbb R,\sigma}(u,u)
\ge C_{\rm coer}(\sigma)\|u\|_{H^s_{M,G}(\mathbb R)}^2,
\qquad
C_{\rm coer}(\sigma)>0 .
\end{equation}
Moreover, if $r>0$, then \eqref{eq:garding} holds with $c_2=0$; that is, the
full-line form is coercive,
\begin{equation}
\label{eq:full_line_coercivity}
\operatorname{Re}a_{\mathbb R}(u,u)
\ge c_1\|u\|_{H^s_{M,G}(\mathbb R)}^2 ,
\end{equation}
and $\sigma=0$ is admissible in \eqref{eq:shifted_coercivity} with
$C_{\rm coer}(0)=c_1$.
\end{theorem}

\begin{proof}[Proof sketch]
Continuity follows from the multiplier representation \eqref{eq:A_symbol},
the norm equivalence of \eqref{eq:HMG_fourier_norm}, and the drift bound
\eqref{eq:drift_duality_bound}.  For the G{\aa}rding inequality, since
$\Gamma(-Y)>0$ and $\cos(Y\pi/2)<0$ for $Y\in(1,2)$, the real part of the
symbol grows like $|\xi|^Y$ at high frequency and is bounded below on compact
sets, so $\operatorname{Re}A(\xi)\ge c(1+|\xi|^2)^s-C_0$ and Plancherel's
theorem gives \eqref{eq:garding}.  The coercive refinement rests on an exact
identity: substituting the definition \eqref{eq:lambda0_def} of $\lambda_0$
into \eqref{eq:A_symbol}, the real part of the symbol is \emph{exactly}
\begin{equation}
\label{eq:symbol_real_part_exact}
\operatorname{Re}A(\xi)
=r
+C\Gamma(-Y)\Bigl[
\bigl(M^Y-\operatorname{Re}(M-i\xi)^Y\bigr)
+\bigl(G^Y-\operatorname{Re}(G+i\xi)^Y\bigr)
\Bigr],
\end{equation}
and by Lemma~\ref{lem:symbol_monotone} both bracketed terms are nonnegative,
so $\operatorname{Re}A(\xi)\ge r>0$ for every $\xi$ when $r>0$.  Combining
this uniform bound at low frequency with the $|\xi|^Y$ growth at high
frequency yields $\operatorname{Re}A(\xi)\ge c_1(1+|\xi|^2)^s$ with no
subtracted constant, hence \eqref{eq:full_line_coercivity}.  The complete
proof, including the explicit constant $c_1$, is given in the supplementary
materials (section~SM1).
\end{proof}

\begin{remark}[The case $r=0$]
\label{rem:r_zero}
If $r=0$, then $\operatorname{Re}A(0)=0$ by
\eqref{eq:symbol_real_part_exact}, so coercivity fails only through the
constant Fourier mode; the G{\aa}rding inequality \eqref{eq:garding} still
holds with an arbitrarily small $c_2>0$.  Since $r>0$ in every experiment and
throughout the calibration bounds of Section~\ref{sec:calibration}, the
coercive case $c_2=0$ is the operative one in this paper, and we call a shift
$\sigma$ \emph{admissible} if $\sigma>c_2$, or $\sigma=0$ when $c_2=0$.  All
shifted statements below hold for any admissible shift.
\end{remark}

\begin{remark}[Algorithmic range versus proved variational range]
\label{rem:algorithmic_vs_theory_range}
The connection coefficient formulae in Section~\ref{sec:bspline_wavelets}
are valid for $Y\in(0,1)\cup(1,2)$.  The proof above, however, is a
coercive shifted variational argument for $Y\in(1,2)$.  When $Y\in(0,1)$,
$s=Y/2<1/2$, and the first-order drift is not bounded on
$H^s\times H^s$ by \eqref{eq:drift_duality_bound}.  A complete theory in the
finite-variation regime should separate the drift from the lower-order jump
part or use a generator/semigroup framework; this is left for future work.
\end{remark}

\subsection{Finite-Domain Truncation and Smooth Exterior Lifting}
\label{ssec:finite_domain_framework}

For a nonlocal CGMY operator, passing from $\mathbb R$ to
$\Omega=(x_{\min},x_{\max})$ is part of the model, not merely a boundary
condition.  The operator acting on points inside $\Omega$ depends on the
chosen exterior extension.  In the theoretical formulation we therefore write
\begin{equation}
\label{eq:continuous_lifting_split}
V(x,\tau)=g(x,\tau)+U(x,\tau),
\end{equation}
where $g$ is a prescribed smooth lifting carrying the far-field behaviour and
$U$ is the homogeneous unknown.  For the put option, $g$ is chosen so that
$g(x,\tau)=Ke^{-r\tau}$ near the left exterior and $g(x,\tau)=0$ near the
right exterior, with a smooth transition across the computational buffer.
We assume
\begin{equation}
\label{eq:lifting_regular_assumption}
g\in L^\infty(0,T;H^q(\mathbb R)),
\qquad
\partial_\tau g\in L^2(0,T;H^{q-Y}(\mathbb R)),
\qquad q\ge p,
\end{equation}
so that the lifting forcing
\begin{equation}
\label{eq:fg_def}
f_g:=-\partial_\tau g-\mathcal A g
\end{equation}
is well-defined in the dual space used below.

Because the homogeneous unknown carries no far-field data --- these are
carried entirely by the lifting $g$ and enter the scheme only through the
forcing --- the natural homogeneous exterior extension is \emph{extension by
zero}.  We fix this convention once and for all: $E_\Omega u=\widetilde u$,
where $\widetilde u$ denotes the extension of $u$ by zero outside $\Omega$.
The corresponding energy space is
\begin{equation}
\label{eq:zero_extension_space}
H^s_{M,G}(\Omega)
:=\bigl\{u\in L^2(\Omega):\widetilde u\in H^s(\mathbb R)\bigr\},
\qquad
\|u\|_{H^s_{M,G}(\Omega)}
:=\|\widetilde u\|_{H^s_{M,G}(\mathbb R)} .
\end{equation}
For $s=Y/2\in(1/2,1)$ this space coincides with
$H^s_0(\Omega)=\{v\in H^s(\Omega):v(x_{\min})=v(x_{\max})=0\}$, with
equivalent norms; see \cite[Theorem~3.33]{McLean2000}.  In particular,
membership requires vanishing endpoint values but no condition on
derivatives, which is the correct homogeneous class for the fractional order
$s<1$.  Throughout the analysis, $H^s_{M,G}(\Omega)$ denotes the space
\eqref{eq:zero_extension_space}.

The finite-domain bilinear form for the homogeneous unknown is
\begin{equation}
\label{eq:finite_domain_bilinear}
a_\Omega(u,w):=\langle \mathcal A\widetilde u,\widetilde w\rangle
=a_{\mathbb R}(\widetilde u,\widetilde w) .
\end{equation}
In the remainder of the paper, whenever no ambiguity is possible, we write
$a(\cdot,\cdot)$ for $a_\Omega(\cdot,\cdot)$ and
$a_\sigma(u,w)=a(u,w)+\sigma(u,w)_{L^2(\Omega)}$.

With the extension operator fixed, the finite-domain G{\aa}rding inequality
is no longer an assumption: it is inherited from the full line.

\begin{theorem}[Finite-domain G\texorpdfstring{\aa}{aa}rding inequality and coercivity]
\label{ass:finite_domain_garding}
Let $Y\in(1,2)$, $s=Y/2$, and assume
\eqref{eq:cgmy_parameter_assumption}.  With the zero-extension convention
\eqref{eq:zero_extension_space}--\eqref{eq:finite_domain_bilinear}, the
finite-domain form is continuous on
$H^s_{M,G}(\Omega)\times H^s_{M,G}(\Omega)$,
\begin{equation}
\label{eq:finite_domain_continuity}
|a_\Omega(u,w)|
\le C_{\rm cont}\|u\|_{H^s_{M,G}(\Omega)}\|w\|_{H^s_{M,G}(\Omega)},
\end{equation}
and satisfies the G{\aa}rding inequality
\begin{equation}
\label{eq:finite_domain_garding}
\operatorname{Re}a_\Omega(u,u)
\ge c_1\|u\|_{H^s_{M,G}(\Omega)}^2
-c_2\|u\|_{L^2(\Omega)}^2 ,
\end{equation}
with the \emph{same} constants $C_{\rm cont}$, $c_1$, $c_2$ as in
Theorem~\ref{thm:wellposed}; in particular the constants are independent of
$h$ and of $\Omega$, and $c_2=0$ whenever $r>0$.  Consequently, for every
admissible shift $\sigma$ (Remark~\ref{rem:r_zero}),
\begin{equation}
\label{eq:finite_domain_shifted_coercivity}
\operatorname{Re}a_{\Omega,\sigma}(u,u)
\ge C_{\rm coer}(\sigma)\|u\|_{H^s_{M,G}(\Omega)}^2 ,
\end{equation}
with $C_{\rm coer}(0)=c_1$ when $c_2=0$.
\end{theorem}

\begin{proof}
For $u,w\in H^s_{M,G}(\Omega)$ one has $\widetilde u,\widetilde w\in
H^s(\mathbb R)$ with $\|\widetilde u\|_{L^2(\mathbb R)}=\|u\|_{L^2(\Omega)}$,
so every estimate follows by applying the full-line inequalities of
Theorem~\ref{thm:wellposed} to $\widetilde u$ and $\widetilde w$ and using
the definition \eqref{eq:zero_extension_space} of the norm; the shifted
estimate follows by adding $\sigma\|u\|_{L^2(\Omega)}^2$.
\end{proof}

\begin{remark}[Scope of the zero-extension convention]
\label{rem:zero_extension_scope}
Theorem~\ref{ass:finite_domain_garding} replaces what was previously an
unverified finite-domain assumption.  The only structural requirement is
that the homogeneous unknown --- and the discrete trial and test functions of
Sections~\ref{ssec:spatial_disc}--\ref{ssec:bc} --- extend by zero, so that
all far-field data enter exclusively through the lifting $g$ and the forcing
$\ell_{\rm bc}$; this is exactly the formulation used by the scheme.  Other
exterior extension conventions, such as the financially natural piecewise
extension \eqref{eq:exterior_extension} applied to the full value, would
require a separate coercivity proof.
\end{remark}

Finally, the interior spline spaces used in the analysis approximate at full
order only functions that vanish near the artificial boundaries.  This is
arranged through the lifting.

\begin{assumption}[Buffer-supported homogeneous unknown]
\label{ass:buffer_lifting}
There is a fixed buffer width $\delta_b>0$, independent of $h$, such that the
lifting $g$ in \eqref{eq:continuous_lifting_split} is chosen with
$U(\cdot,\tau)=V(\cdot,\tau)-g(\cdot,\tau)$ and $\partial_\tau U(\cdot,\tau)$
vanishing on the buffer strips
$\Omega_b:=(x_{\min},x_{\min}+\delta_b)\cup(x_{\max}-\delta_b,x_{\max})$
for all $\tau\in[0,T]$.
\end{assumption}

\begin{remark}[Constructing a buffer-supported lifting]
\label{rem:buffer_lifting_construction}
Assumption~\ref{ass:buffer_lifting} can always be met by a cutoff
construction: because the solution approaches its far-field values at the
exponentially tempered rate of Proposition~\ref{prop:domain_error}, the
required modification changes $g$ only by exponentially small, smooth terms
on the buffer strips and preserves the regularity
\eqref{eq:lifting_regular_assumption}.  The construction is detailed in the
supplementary materials (section~SM2).
\end{remark}

\subsection{American Put Obstacle Problem}
\label{ssec:american_obstacle_model}

The same backward-time CGMY operator also leads to an American put problem by
adding the early-exercise constraint.  Let $\psi_A(x)=(K-e^x)^+$.
For an American put, the value must remain above the payoff and, in the
continuation region, it satisfies the same linear equation as the European
put.  With the sign convention of \eqref{eq:backward_time_strong}, the
continuous complementarity formulation is
\begin{equation}
\label{eq:american_continuous_lcp}
V(x,\tau)\ge \psi_A(x),\qquad
\partial_\tau V(x,\tau)+\mathcal A V(x,\tau)\ge 0,
\end{equation}
with
\begin{equation}
\label{eq:american_complementarity}
\bigl(V(x,\tau)-\psi_A(x)\bigr)
\bigl(\partial_\tau V(x,\tau)+\mathcal A V(x,\tau)\bigr)=0 .
\end{equation}
Equivalently, if
\begin{equation}
\label{eq:american_admissible_set}
\mathcal K:=\{W\in H^s_{M,G}(\Omega): W\ge \psi_A\text{ a.e. on }\Omega\},
\end{equation}
then the variational inequality is
\begin{equation}
\label{eq:american_vi}
\bigl\langle \partial_\tau V+\mathcal A V, W-V\bigr\rangle \ge 0,
\qquad W\in\mathcal K .
\end{equation}
The formulation \eqref{eq:american_admissible_set}--\eqref{eq:american_vi}
is stated formally for the full value; on the truncated domain, the far-field
data are handled by the same lifting device as in the European case, with
$H^s_{M,G}(\Omega)$ understood for the lifted unknown as in
Section~\ref{ssec:finite_domain_framework}.  No well-posedness or
convergence theory is claimed for the obstacle problem here; the American
solver is presented as a numerical extension
(Section~\ref{ssec:american_lcp}).
The initial condition in backward time is again $V(x,0)=\psi_A(x)$.
The finite-domain exterior data differ from the European put only at the
left boundary: deep in the money, immediate exercise is optimal, so the
American far-field behaviour is
\begin{equation}
\label{eq:american_farfield}
V(x,\tau)\sim K-e^x\quad (x\to -\infty),
\qquad
V(x,\tau)\sim 0\quad (x\to +\infty).
\end{equation}
Thus the left exterior value is not discounted by $e^{-r\tau}$, unlike the
European put boundary in \eqref{eq:put_farfield}.  The numerical
implementation in Section~\ref{ssec:american_lcp} uses this obstacle
formulation with the same B-spline connection-coefficient matrix assembled
for the European problem.


\section{B-Spline Scaling Basis and Tempered Connection Coefficients}
\label{sec:bspline_wavelets}

\subsection{B-Spline Scaling Functions}
\label{ssec:bspline_basics}

This section derives the connection coefficients used by the spatial
Galerkin discretisation.  The discretisation below uses the cardinal
B-spline \emph{scaling space} at a fixed resolution level.  This distinction
is important: the compactly supported scaling functions provide the
approximation space, while genuine spline wavelets may later be used for
multilevel compression.

\begin{definition}[Cardinal B-spline]
The cardinal B-spline of order $p\geq 1$ is defined recursively by
\begin{equation}
\label{eq:bspline_rec}
N_p(x) = \frac{x}{p-1}\,N_{p-1}(x)
       + \frac{p-x}{p-1}\,N_{p-1}(x-1),
\qquad N_1(x)=\mathbf{1}_{[0,1)}(x).
\end{equation}
Equivalently,
\begin{equation}
\label{eq:bspline_truncated_power}
N_p(x)=\frac{1}{\Gamma(p)}
\sum_{m=0}^{p}(-1)^m\binom{p}{m}(x-m)_+^{p-1},
\qquad (x)_+:=\max\{x,0\}.
\end{equation}
The function $N_p$ is supported on $[0,p]$ and belongs to
$C^{p-2}(\mathbb{R})$ for $p\geq 2$.
\end{definition}

At resolution level $j\geq 0$, set
\begin{equation}
\label{eq:scaling_function}
\bsphi_{j,k}(x)=2^{j/2}N_p(2^j x-k),
\qquad k\in\mathbb{Z},
\end{equation}
and define
\begin{equation}
V_j=\overline{\mathrm{span}}\{\bsphi_{j,k}:k\in\mathbb{Z}\}.
\end{equation}
The Fourier transform of $N_p$ is
\begin{equation}
\label{eq:bspline_fourier}
\widehat N_p(\xi)
=\left(\frac{1-e^{-i\xi}}{i\xi}\right)^p,
\end{equation}
and therefore
\begin{equation}
\label{eq:phi_fourier}
\widehat{\bsphi}_{j,k}(\xi)
=2^{-j/2}e^{-ik2^{-j}\xi}\widehat N_p(2^{-j}\xi).
\end{equation}
The approximation property used later is
\begin{equation}
\label{eq:bspline_approximation_order}
\inf_{v\in V_j}\|u-v\|_{H^s(\mathbb{R})}
\leq C2^{-j(p-s)}\|u\|_{H^p(\mathbb{R})},
\qquad 0\leq s\leq p.
\end{equation}

\begin{remark}[Scaling functions versus wavelets]
The scaling function $N_p$ does \emph{not} have vanishing moments; indeed
$\int_{\mathbb{R}}N_p(x)\,dx=1$.  Vanishing moments are a property of the
associated spline wavelets, such as Battle--Lemari\'e or compactly supported
biorthogonal spline wavelets.  The present paper uses the scaling space
$V_j$ for the Galerkin approximation.  Consequently, the exact Toeplitz
structure derived below follows from translation invariance, not from
wavelet cancellation.  A genuinely multilevel wavelet compression would
require a separate basis transformation and a separate compressibility proof.
\end{remark}

\subsection{Tempered Riemann--Liouville Operators}
\label{ssec:tempered_symbols}

The CGMY FPDE contains exponentially conjugated Riemann--Liouville
operators.  We define the left- and right-tempered operators by
\begin{equation}
\label{eq:tempered_left_operator}
\mathcal{D}_{G}^{-}u
:=e^{-Gx}\,{}_{-\infty}D_x^Y\bigl(e^{Gx}u\bigr),
\end{equation}
and
\begin{equation}
\label{eq:tempered_right_operator}
\mathcal{D}_{M}^{+}u
:=e^{Mx}\,{}_{x}D_{\infty}^Y\bigl(e^{-Mx}u\bigr).
\end{equation}
The outside exponential factors are part of the operator.  Omitting them
would destroy the translation-invariant tempered symbol and would lead to an
incorrect connection matrix.

\begin{proposition}[Fourier symbols of the tempered operators]
\label{prop:tempered_symbols}
Let $u$ be compactly supported, or more generally let $u$ have a Fourier
transform that admits the required analytic shifts in the strips
$|\operatorname{Im}\xi|\leq \max\{G,M\}$. Let the fractional power be taken on the
principal branch.  Then
\begin{equation}
\label{eq:left_tempered_symbol}
\widehat{\mathcal{D}_{G}^{-}u}(\xi)
=(G+i\xi)^Y\widehat u(\xi),
\end{equation}
and
\begin{equation}
\label{eq:right_tempered_symbol}
\widehat{\mathcal{D}_{M}^{+}u}(\xi)
=(M-i\xi)^Y\widehat u(\xi).
\end{equation}
\end{proposition}

The proof --- a Fourier shift argument for the exponential conjugations
combined with the Riemann--Liouville symbols
$\widehat{{}_{-\infty}D_x^Y f}(\xi)=(i\xi)^Y\widehat f(\xi)$ and
$\widehat{{}_{x}D_{\infty}^Y f}(\xi)=(-i\xi)^Y\widehat f(\xi)$ --- is given in
the supplementary materials (section~SM3).

\subsection{Toeplitz Connection Coefficients}
\label{ssec:connection_coefficients}

The connection matrices are now defined using the complete tempered
operators \eqref{eq:tempered_left_operator}--\eqref{eq:tempered_right_operator}.
At resolution level $j$, set
\begin{align}
\label{eq:Aminus_correct}
(\Aminus)_{kk'}
&:=\int_{\mathbb{R}}\bsphi_{j,k}(x)
\,\mathcal{D}_{G}^{-}\bsphi_{j,k'}(x)\,dx,\\
\label{eq:Aplus_correct}
(\Aplus)_{kk'}
&:=\int_{\mathbb{R}}\bsphi_{j,k}(x)
\,\mathcal{D}_{M}^{+}\bsphi_{j,k'}(x)\,dx.
\end{align}
The financial prefactor $C\Gamma(-Y)$ is applied later when the full CGMY
stiffness matrix is assembled.

\begin{theorem}[Tempered B-spline connection coefficients]
\label{thm:tempered_connection_coefficients}
On the full line, the entries of $\Aminus$ and $\Aplus$ depend only on the
index difference $\ell=k-k'$.  More precisely,
\begin{equation}
\label{eq:Aminus_toeplitz_correct}
(\Aminus)_{kk'}=2^{jY}\Lambda_G^-(k-k';G2^{-j}),
\end{equation}
where
\begin{equation}
\label{eq:LambdaGminus}
\Lambda_G^-(\ell;\delta)
=\frac{1}{2\pi}\int_{\mathbb{R}}
(\delta+i\eta)^Y|\widehat N_p(\eta)|^2e^{i\ell\eta}\,d\eta,
\qquad \delta=G2^{-j},
\end{equation}
and
\begin{equation}
\label{eq:Aplus_toeplitz_correct}
(\Aplus)_{kk'}=2^{jY}\Lambda_M^+(k-k';M2^{-j}),
\end{equation}
where
\begin{equation}
\label{eq:LambdaMplus}
\Lambda_M^+(\ell;\delta)
=\frac{1}{2\pi}\int_{\mathbb{R}}
(\delta-i\eta)^Y|\widehat N_p(\eta)|^2e^{i\ell\eta}\,d\eta,
\qquad \delta=M2^{-j}.
\end{equation}
Thus the infinite-domain connection matrices are Toeplitz.  The factor
$2^{jY}$ is the correct scaling of a fractional derivative of order $Y$.
\end{theorem}

\begin{proof}
By Plancherel's identity and Proposition~\ref{prop:tempered_symbols},
\begin{align*}
(\Aminus)_{kk'}
&=\frac{1}{2\pi}\int_{\mathbb{R}}
\overline{\widehat{\bsphi}_{j,k}(\xi)}
(G+i\xi)^Y
\widehat{\bsphi}_{j,k'}(\xi)\,d\xi .
\end{align*}
Using \eqref{eq:phi_fourier},
\begin{align*}
\overline{\widehat{\bsphi}_{j,k}(\xi)}
\widehat{\bsphi}_{j,k'}(\xi)
=2^{-j}e^{i(k-k')2^{-j}\xi}
|\widehat N_p(2^{-j}\xi)|^2.
\end{align*}
With the change of variables $\eta=2^{-j}\xi$, we obtain
\begin{align*}
(\Aminus)_{kk'}
&=\frac{1}{2\pi}\int_{\mathbb{R}}
(G+i2^j\eta)^Y|\widehat N_p(\eta)|^2
 e^{i(k-k')\eta}\,d\eta  \\
&=2^{jY}\frac{1}{2\pi}\int_{\mathbb{R}}
(G2^{-j}+i\eta)^Y|\widehat N_p(\eta)|^2
 e^{i(k-k')\eta}\,d\eta .
\end{align*}
This proves \eqref{eq:Aminus_toeplitz_correct}--\eqref{eq:LambdaGminus}.
The proof of \eqref{eq:Aplus_toeplitz_correct}--\eqref{eq:LambdaMplus} is identical,
using the symbol $(M-i\xi)^Y$.
\end{proof}

\begin{remark}[Boundary corrections]
The exact Toeplitz structure holds for the translation-invariant full-line
operator and the infinite uniform family $\{\bsphi_{j,k}\}_{k\in\mathbb{Z}}$.
On a truncated interval, the interior block remains Toeplitz, while boundary
conditions and modified endpoint basis functions introduce finite boundary
corrections.  These boundary corrections should be treated separately from
tempering; they are not an ``approximately Toeplitz'' tempering effect.  In
the scheme of Section~\ref{sec:galerkin_scheme} the issue is sidestepped
entirely: the homogeneous trial and test space uses only interior-supported
splines, so its Galerkin matrix is \emph{exactly} a principal Toeplitz
submatrix (Lemma~\ref{lem:conformity}), and all exterior data enter through
the lifting and forcing.
\end{remark}

\subsection{Exact Untempered Formulae and Tempered Evaluation}
\label{ssec:explicit_formulas}

The Fourier formulae above are the most stable way to define the tempered
coefficients.  In the untempered limit $\delta=0$, the coefficients also
admit closed-form finite-difference expressions.

\begin{proposition}[Closed form for the untempered left coefficients]
\label{prop:untempered_closed_form}
Let $p\geq 1$ and $0<Y<2$.  For $\delta=0$,
\begin{equation}
\label{eq:Lambda_untempered_general}
\Lambda_0^-(\ell)
=\frac{1}{\Gamma(2p-Y)}
\sum_{m=0}^{2p}(-1)^m\binom{2p}{m}
(\ell+p-m)_+^{2p-Y-1}.
\end{equation}
Moreover,
\begin{equation}
\label{eq:Lambda_untempered_right}
\Lambda_0^+(\ell)=\Lambda_0^-(-\ell).
\end{equation}
\end{proposition}

The proof --- via the B-spline autocorrelation identity
$R_p(t)=N_{2p}(t+p)$ and the Riemann--Liouville derivative of truncated
powers --- is given in the supplementary materials (section~SM3).

For the two low-order cases used most often in computations,
\begin{align}
\label{eq:Lambda_p2_correct}
p=2:\qquad
\Lambda_0^-(\ell)
&=\frac{1}{\Gamma(4-Y)}
\sum_{m=0}^{4}(-1)^m\binom{4}{m}(\ell+2-m)_+^{3-Y},\\
\label{eq:Lambda_p3_correct}
p=3:\qquad
\Lambda_0^-(\ell)
&=\frac{1}{\Gamma(6-Y)}
\sum_{m=0}^{6}(-1)^m\binom{6}{m}(\ell+3-m)_+^{5-Y}.
\end{align}
These formulae replace the informal expressions involving absolute values
and unspecified signs.  The positive-part convention is essential because
left and right Riemann--Liouville derivatives have one-sided algebraic tails.

For $\delta>0$, the exact coefficients are defined by
\eqref{eq:LambdaGminus}--\eqref{eq:LambdaMplus}.  They can be precomputed
once at a given resolution using high-order quadrature, or by evaluating the
Fourier integral on an oversampled grid and using an inverse FFT.  The
untempered formulae \eqref{eq:Lambda_p2_correct}--\eqref{eq:Lambda_p3_correct}
provide useful checks on the implementation as $\delta\to0$.

\subsection{Decay, FFT Application, and Optional Truncation}
\label{ssec:sparsity}

The matrices generated by the scaling basis are dense.  Their computational
advantage is Toeplitz structure, not true sparsity.  The Toeplitz structure
permits fast matrix-vector products, while banded truncation is an optional
approximation whose effect must be controlled separately.

\begin{lemma}[Algebraic decay of connection coefficients]
\label{lem:decay}
Assume $p\geq 2$ and fix $\delta_0>0$.  There exists a constant
$C=C(p,Y,\delta_0)$ such that, for all $0\leq\delta\leq\delta_0$ and all
$\ell\in\mathbb{Z}$,
\begin{equation}
\label{eq:coefficient_decay_correct}
|\Lambda_G^-(\ell;\delta)|+|\Lambda_M^+(\ell;\delta)|
\leq C(1+|\ell|)^{-(1+Y)}.
\end{equation}
For $\delta>0$, one may obtain sharper one-sided exponentially tempered
bounds, but the uniform algebraic estimate \eqref{eq:coefficient_decay_correct}
is sufficient for the discretisation analysis.
\end{lemma}

The proof --- a finite-difference tail estimate for the untempered
coefficients combined with a uniform Fourier-multiplier argument for
$\delta\in[0,\delta_0]$ --- is given in the supplementary materials
(section~SM3).

\begin{proposition}[Fast Toeplitz application]
\label{prop:toeplitz_matvec}
Let $T_N$ be the $N\times N$ Toeplitz matrix with entries
$(T_N)_{kk'}=\Lambda(k-k')$, where $\Lambda$ denotes either
$\Lambda_G^-(\cdot;G2^{-j})$ or $\Lambda_M^+(\cdot;M2^{-j})$.  Then
$T_N\mathbf c$ can be computed in $O(N\log N)$ operations by embedding
$T_N$ in a circulant matrix of size at least $2N$ and applying the FFT.
Consequently, the fractional CGMY part can be applied in $O(N\log N)$ time
without replacing the dense Toeplitz matrix by a sparse banded matrix.
\end{proposition}

\begin{proof}
Standard circulant embedding: the Toeplitz data are placed in a circulant
vector of length $2N$, converting the product into a circular convolution
evaluated by one forward FFT, one pointwise product, and one inverse FFT.
\end{proof}

\begin{proposition}[Optional finite-band truncation]
\label{prop:optional_truncation}
Let $T_N^{(B)}$ be obtained from $T_N$ by setting entries with
$|k-k'|>B$ to zero.  Then
\begin{equation}
\label{eq:relative_truncation_tail}
\|T_N-T_N^{(B)}\|_{\ell^2\to\ell^2}
\leq C B^{-Y}.
\end{equation}
After including the fractional scaling factor, the corresponding stiffness
matrix perturbation is bounded by $C2^{jY}B^{-Y}$ in coefficient
$\ell^2$-norm.
\end{proposition}

\begin{proof}
By Lemma~\ref{lem:decay},
\begin{equation*}
\sum_{|\ell|>B}|\Lambda(\ell)|
\leq C\sum_{\ell>B}\ell^{-(1+Y)}
\leq C B^{-Y}.
\end{equation*}
Young's inequality for discrete convolution gives
\eqref{eq:relative_truncation_tail}.  Multiplication by $2^{jY}$ gives the
scaled stiffness estimate.
\end{proof}

\begin{remark}[Scope of the truncation result]
The rigorous discretisation results below are based on the full Toeplitz
operator applied by FFT; Proposition~\ref{prop:optional_truncation} justifies
finite-band truncation only as a controllable implementation approximation at
a prescribed tolerance.  It does \emph{not} show that a sublinear bandwidth
$B=N^\beta$, $\beta<1$, preserves the full Galerkin rate, which would require
a genuine multilevel wavelet compressibility theorem.
\end{remark}


\section{The B-Spline Galerkin Scheme}
\label{sec:galerkin_scheme}

\subsection{Backward-Time Operator and Sign Convention}
\label{ssec:backward_operator}

The Galerkin scheme discretises exactly the backward-time operator
$\mathcal A$ defined in \eqref{eq:backward_operator_A}.  Thus the continuous
weak form, the semidiscrete matrix, the shifted projection, and the
Crank--Nicolson stability analysis all use the same bilinear form
\begin{equation}
\label{eq:scheme_bilinear_same_a}
a(u,w)=\langle \mathcal A u,w\rangle .
\end{equation}
The sign convention is therefore
\begin{equation}
\partial_\tau V+\mathcal A V=0,
\qquad V(x,0)=\psi(x),
\end{equation}
with $\mathcal A$ given by
\eqref{eq:backward_operator_A}.  This is the operator whose discrete matrix is
assembled below.

\subsection{Physical B-Spline Space and Spatial Discretisation}
\label{ssec:spatial_disc}

Let
\begin{equation}
\Omega=(x_{\min},x_{\max}),
\qquad
h=\frac{x_{\max}-x_{\min}}{N}.
\end{equation}
We use cardinal B-splines on the uniform physical grid and restrict them to
$\Omega$ for the mass and derivative matrices.  The active index set is
\begin{equation}
\label{eq:active_index_set}
\mathcal I_h
:=\bigl\{k\in\mathbb Z:
\operatorname{supp}\phi_{h,k}\cap\Omega\ne\emptyset\bigr\},
\end{equation}
where
\begin{equation}
\label{eq:physical_bspline_basis}
\phi_{h,k}(x)
:=h^{-1/2}N_p\!\left(\frac{x-x_{\min}}{h}-k\right),
\qquad k\in\mathcal I_h .
\end{equation}
The factor $h^{-1/2}$ gives an $L^2$-normalised basis.  The
\emph{homogeneous} trial and test space uses only the interior-supported
splines,
\begin{equation}
\label{eq:interior_index_set}
\mathcal I_h^0
:=\bigl\{k\in\mathbb Z:
\operatorname{supp}\phi_{h,k}\subset[x_{\min},x_{\max}]\bigr\}
=\{0,1,\dots,N-p\},
\end{equation}
\begin{equation}
\label{eq:Vh0_def}
V_h^0:=\operatorname{span}\{\phi_{h,k}:k\in\mathcal I_h^0\}.
\end{equation}
Each element of $V_h^0$ is supported in $\overline\Omega$, so it coincides
with its own zero extension; this makes $V_h^0$ conforming for the
zero-extension energy space of Section~\ref{ssec:finite_domain_framework}
(Lemma~\ref{lem:conformity} below).  The boundary-crossing splines
$k\in\mathcal I_h\setminus\mathcal I_h^0$ are used only to represent the
discrete lifting; in the theoretical formulation, the physical approximation
is
\begin{equation}
\label{eq:affine_lifted_space}
V_h^g(\tau)=g_h(\tau)+V_h^0,
\end{equation}
where $g_h$ is the discrete lifting from Section~\ref{ssec:bc}.  Endpoint
basis functions are not assumed to be periodic.  All far-field and lifting
data enter the homogeneous problem only through forcing terms; the
homogeneous Galerkin matrix requires no boundary-modified rows.

The physical-grid version of Theorem~\ref{thm:tempered_connection_coefficients}
is
\begin{equation}
\label{eq:physical_left_coeff}
(K_G^-)_{kk'}
=
\int_{\mathbb R}\phi_{h,k}(x)\mathcal D_G^-\phi_{h,k'}(x)\,dx
=
h^{-Y}\Lambda_G^-(k-k';Gh),
\end{equation}
for the full-line interior block, and
\begin{equation}
\label{eq:physical_right_coeff}
(K_M^+)_{kk'}
=
\int_{\mathbb R}\phi_{h,k}(x)\mathcal D_M^+\phi_{h,k'}(x)\,dx
=
h^{-Y}\Lambda_M^+(k-k';Mh).
\end{equation}
Thus the fractional stiffness matrices scale like $h^{-Y}$ on the physical
mesh.  For the homogeneous space $V_h^0$ these formulae are exact --- no
boundary-modified rows arise (Lemma~\ref{lem:conformity} below) --- and all
exterior and lifting data enter as forcing terms, discussed in
Section~\ref{ssec:bc}.

For test functions $w_h\in V_h^0$,
the homogeneous interior Galerkin form is
\begin{equation}
\label{eq:ah_def}
a_h(u_h,w_h):=(\mathcal A u_h,w_h)_{L^2(\Omega)}.
\end{equation}

\begin{lemma}[Conformity and exactness of the discrete Galerkin form]
\label{lem:conformity}
Let $p\ge2$ and $s=Y/2\in(1/2,1)$.  Then
$V_h^0\subset H^s_{M,G}(\Omega)$, and for all $u_h,w_h\in V_h^0$,
\begin{equation}
\label{eq:ah_equals_aOmega}
a_h(u_h,w_h)=a_\Omega(u_h,w_h)
=a_{\mathbb R}(u_h,w_h).
\end{equation}
Consequently, the homogeneous fractional stiffness blocks are exactly the
principal submatrices of the full-line Toeplitz connection-coefficient
matrices indexed by $\mathcal I_h^0$ --- themselves Toeplitz --- and no
boundary corrections arise in the homogeneous problem.
\end{lemma}

The proof is given in the supplementary materials (section~SM1):
interior splines coincide with their own zero extensions, so the discrete
form is the restriction of the full-line form, and a principal submatrix of a
Toeplitz matrix with contiguous indices is Toeplitz.
With the coefficient expansion
\begin{equation}
U_h(x,\tau)=\sum_{k\in\mathcal I_h^0}c_k(\tau)\phi_{h,k}(x),
\end{equation}
the matrices are
\begin{align}
\label{eq:Mh_def}
(M_h)_{kk'}&=(\phi_{h,k'},\phi_{h,k})_{L^2(\Omega)},\\
\label{eq:Dh_def}
(D_h)_{kk'}&=(\phi'_{h,k'},\phi_{h,k})_{L^2(\Omega)},\\
\label{eq:KG_def}
(K_G^-)_{kk'}&=(\mathcal D_G^-\phi_{h,k'},\phi_{h,k})_{L^2(\Omega)},\\
\label{eq:KM_def}
(K_M^+)_{kk'}&=(\mathcal D_M^+\phi_{h,k'},\phi_{h,k})_{L^2(\Omega)}.
\end{align}
The semidiscrete operator matrix is therefore
\begin{equation}
\label{eq:Ah_def}
A_h
=
-(r-\omega_S)D_h
-C\Gamma(-Y)K_M^+
-C\Gamma(-Y)K_G^-
+\lambda_0 M_h .
\end{equation}
The mass matrix $M_h$ and derivative matrix $D_h$ are banded because the
B-spline basis has compact support.  The interior blocks of $K_G^-$ and
$K_M^+$ are Toeplitz and are applied by FFT using
Proposition~\ref{prop:toeplitz_matvec}.  Equation \eqref{eq:Ah_def} also
corrects the reaction term: the coefficient is $\lambda_0$, not merely $r$.

\subsection{Boundary Lifting and Exterior Contribution}
\label{ssec:bc}

For a European put, the far-field behaviour is
\begin{equation}
\label{eq:put_farfield}
V(x,\tau)\sim Ke^{-r\tau}\quad (x\to-\infty),
\qquad
V(x,\tau)\sim0\quad (x\to+\infty).
\end{equation}
Because $\mathcal D_G^-$ and $\mathcal D_M^+$ are nonlocal, pointwise
Dirichlet data at $x_{\min}$ and $x_{\max}$ do not define the truncated
operator.  The exterior extension is part of the numerical model.

The rigorous analysis uses the smooth lifting already introduced in
\eqref{eq:continuous_lifting_split}.  At the discrete level we write
\begin{equation}
\label{eq:lifting_split}
V_h(x,\tau)=g_h(x,\tau)+U_h(x,\tau),
\qquad U_h(\tau)\in V_h^0,
\end{equation}
where $g_h$ is a spline or projected approximation of the smooth lifting $g$.
The homogeneous unknown $U_h$ is the variable advanced in time.  Its
semidiscrete weak problem is
\begin{equation}
\label{eq:semidiscrete_forced_weak}
(\partial_\tau U_h,w_h)_{L^2(\Omega)}+a_h(U_h,w_h)
=\ell_{\rm bc}(w_h;\tau),
\qquad w_h\in V_h^0,
\end{equation}
where
\begin{equation}
\label{eq:bc_functional}
\ell_{\rm bc}(w_h;\tau)
=-(\partial_\tau g_h,w_h)_{L^2(\Omega)}-a_h(g_h,w_h)
+\ell_{\rm ext}(w_h;\tau).
\end{equation}
The term $\ell_{\rm ext}$ collects the part of the nonlocal operators caused
by the prescribed values outside $\Omega$ after the lifting and homogeneous
extension have been selected.

For implementation, one may use the financially natural piecewise exterior
extension
\begin{equation}
\label{eq:exterior_extension}
V_{\rm ext}(x,\tau)=
\begin{cases}
Ke^{-r\tau}, & x<x_{\min},\\
V_h(x,\tau), & x\in\Omega,\\
0, & x>x_{\max}.
\end{cases}
\end{equation}
This piecewise extension is simple and often accurate on a sufficiently large
truncated interval.  However, the convergence proofs are stated for the
smooth-lifting formulation because the fractional derivatives with
$Y\in(1,2)$ can detect jumps at the artificial boundaries.  Thus the
piecewise extension should be regarded as an implementation choice whose
boundary effect is included in the exterior/truncation error.

In coefficient form, with
$U_h(x,\tau)=\sum_{k\in\mathcal I_h^0}c_k(\tau)\phi_{h,k}(x)$,
\eqref{eq:semidiscrete_forced_weak} becomes
\begin{equation}
\label{eq:matrix_ode_forced}
M_h\dot{\mathbf c}(\tau)+A_h\mathbf c(\tau)=\mathbf f_{\rm bc}(\tau),
\end{equation}
where $\mathbf f_{\rm bc}$ is the coefficient vector associated with
\eqref{eq:bc_functional}.  If homogeneous far-field conditions are used in a
manufactured-solution test, then $g_h=0$ and $\mathbf f_{\rm bc}=0$.

\begin{proposition}[Domain-truncation estimate]
\label{prop:domain_error}
Let $\Delta_L=\log K-x_{\min}$ and
$\Delta_R=x_{\max}-\log K$.  Under the exponentially tempered CGMY tails and
under the smooth-lifting assumptions \eqref{eq:lifting_regular_assumption},
the truncation error satisfies the qualitative bound
\begin{equation}
\label{eq:two_sided_domain_error}
\sup_{x\in\Omega}|V(x,\tau)-V^{\rm trunc}(x,\tau)|
\leq
C_T\left(e^{-G\Delta_L}+e^{-M\Delta_R}\right),
\qquad 0\leq\tau\leq T,
\end{equation}
where $C_T$ depends on the maturity, payoff growth, lifting regularity, and
CGMY parameters but not on $h$.
\end{proposition}

The proof, which bounds the omitted exterior contributions by the
exponentially tempered tail integrals of the CGMY L\'evy density and
transfers them to the option value by the stability of the shifted
finite-domain problem, is given in the supplementary materials
(section~SM3).

Thus a balanced choice is
\begin{equation}
\label{eq:domain_choice}
\Delta_L\approx \frac{10}{G},
\qquad
\Delta_R\approx \frac{10}{M},
\end{equation}
which makes both exterior tail contributions approximately of size $e^{-10}$.

\subsection{Projection of the Initial Payoff}
\label{ssec:initial_projection}

The initial condition is imposed by $L^2$ projection.  For the lifted unknown
$U_h=V_h-g_h$, set
\begin{equation}
\label{eq:initial_projection}
M_h\mathbf c^0=\mathbf b^0,
\qquad
b_k^0=\bigl(\psi-g_h(\cdot,0),\phi_{h,k}\bigr)_{L^2(\Omega)}.
\end{equation}
This is preferable to pointwise interpolation in the convergence analysis.
For nonsmooth payoffs such as the European put, the kink at the strike may
reduce the observed global convergence rate unless the payoff is smoothed,
the grid is locally refined near the strike, or a manufactured smooth solution
is used for verification.

\subsection{Crank--Nicolson Time-Stepping}
\label{ssec:cn_time}

Let $\Delta\tau=T/N_T$ and
$\mathbf c^n\approx\mathbf c(n\Delta\tau)$.  Applying the trapezoidal rule to
\eqref{eq:matrix_ode_forced} gives
\begin{equation}
\label{eq:cn}
\left(M_h+\frac{\Delta\tau}{2}A_h\right)\mathbf c^{n+1}
=
\left(M_h-\frac{\Delta\tau}{2}A_h\right)\mathbf c^n
+\frac{\Delta\tau}{2}\left(\mathbf f_{\rm bc}^{n+1}+\mathbf f_{\rm bc}^{n}\right).
\end{equation}
The matrices are time-independent when the CGMY parameters are fixed.  There
are three natural implementation routes.

\begin{enumerate}[label=(\alph*), leftmargin=2.5em]
\item \textbf{Full Toeplitz, matrix-free.}
Use GMRES, BiCGSTAB, or another Krylov method for the linear system in
\eqref{eq:cn}.  Each fractional matrix-vector product costs $O(N\log N)$ by
FFT.  The total cost is $O(m_NN\log N)$ per time step, where $m_N$ is the
number of Krylov iterations.

\item \textbf{Toeplitz/circulant preconditioning.}
Use a circulant or banded approximation of the Toeplitz fractional operator as
a preconditioner.  This retains the exact Toeplitz operator in the residual
while reducing the Krylov iteration count.

\item \textbf{Finite-band approximation.}
Replace $K_G^-$ and $K_M^+$ by bandwidth-$B$ truncations, assemble the sparse
matrix, and factor $M_h+(\Delta\tau/2)A_h$ once.  This is often the simplest
implementation, but the band truncation is then an additional perturbation and
must be included in the numerical error budget.
\end{enumerate}

The convergence theory in Section~\ref{sec:analysis} is stated for the
untruncated Galerkin operator.  The finite-band variant is treated as an
implementation approximation.

\paragraph{Shifted scheme and the choice of $\sigma$.}
The stability analysis of Section~\ref{ssec:stability} is carried out not for
\eqref{eq:cn} directly but for the exponentially shifted unknown
\begin{equation}
\label{eq:shifted_unknown_algo}
W(\tau):=e^{-\sigma\tau}V(\tau),
\end{equation}
where $\sigma$ is any admissible shift in the sense of
Remark~\ref{rem:r_zero}.  The shifted variable solves the semidiscrete system
$M_h\dot{\mathbf d}+(A_h+\sigma M_h)\mathbf d=e^{-\sigma\tau}\mathbf f_{\rm bc}$,
obtained exactly from \eqref{eq:matrix_ode_forced}, and applying the
trapezoidal rule to \emph{this} system yields the shifted Crank--Nicolson
scheme \eqref{eq:shifted_cn_scheme}, from which
$V_h^n=e^{\sigma\tau_n}W_h^n$.  Because the trapezoidal rule does not commute
with the exponential change of variables, the shifted scheme with $\sigma>0$
is \emph{not} the algebraic image of \eqref{eq:cn}: advancing $W_h$ is a
genuine, mild modification that adds the diagonal shift $\sigma M_h$ and
rescales the output.  The choice of $\sigma$ is settled by the theory: by
Theorem~\ref{ass:finite_domain_garding} the finite-domain form is coercive
($c_2=0$) whenever $r>0$, which holds in every experiment reported here
($r=0.1$), so $\sigma=0$ is admissible and the shifted scheme
\emph{coincides identically} with the plain iteration \eqref{eq:cn} that the
code runs.  The stability and convergence theorems of
Section~\ref{sec:analysis} therefore govern exactly the algorithm as
implemented, with no diagonal shift and no output rescaling; the general
$\sigma>c_2$ formulation is retained because it also covers the boundary case
$r=0$ and extensions without coercivity.  The nonlinear American solver of
Section~\ref{ssec:american_lcp} runs the same plain scheme inside the
complementarity iteration and remains a numerical extension, outside the
scope of the linear stability theorem.

\begin{remark}[Rannacher startup]
\label{rem:rannacher}
The reported computations replace the first Crank--Nicolson step by two
implicit-Euler half-steps (Rannacher startup) to damp the nonsmooth-payoff
modes \cite{Rannacher1984,GilesCarter2006}.  This does not affect the scope
of the analysis: the implicit-Euler step is the $\theta=1$ member of the
$\theta$-family, whose energy identity replaces the midpoint value
$\bar{\mathbf d}^{n+1/2}$ in \eqref{eq:shifted_cn_scheme} by the endpoint
value $\mathbf d^{n+1}$ and yields the same shifted-energy estimate with
strictly stronger dissipation under the coercivity of
Theorem~\ref{ass:finite_domain_garding}; and the two startup steps introduce
a local $O(\Delta\tau^2)$ perturbation that leaves the global second-order
consistency of Theorem~\ref{thm:convergence} unchanged
\cite{GilesCarter2006}.
\end{remark}

\subsection{American Option LCP and Projected SOR}
\label{ssec:american_lcp}

For American puts, the spatial operator and mass matrix are exactly the same
as in the European solver.  The difference is that each Crank--Nicolson step
is replaced by a finite-dimensional linear complementarity problem.  For
clarity, this subsection writes the formulation for the full value
coefficients $\mathbf v^n$; in the lifted formulation the same equations are
applied to the homogeneous unknown after moving the lifting terms into the
right-hand side.

Let $\mathbf g$ denote the Galerkin representation of the payoff obstacle.
Because the B-spline coefficients are not point values, the obstacle is
imposed through the $L^2$ projection
\begin{equation}
\label{eq:american_obstacle_projection}
M_h\mathbf g=\mathbf b_\psi,
\qquad
(b_\psi)_k=(\psi_A,\phi_{h,k})_{L^2(\Omega)} ,
\end{equation}
so that $\mathbf g=M_h^{-1}\mathbf b_\psi$ are the coefficients of the $L^2$
projection $\psi_{A,h}:=P_h\psi_A=\sum_k g_k\phi_{h,k}$ of the payoff onto
$V_h$. This projection is required for an $L^2$-normalised B-spline basis,
whose coefficients scale like values multiplied by $h^{1/2}$ and cannot be
constrained by sampling the payoff at grid points.

\begin{remark}[Discrete obstacle versus true obstacle]
\label{rem:discrete_obstacle}
The discrete variational inequality is posed with respect to the
\emph{projected} obstacle $\psi_{A,h}=P_h\psi_A$, not the exact payoff
$\psi_A$. The coefficientwise constraint $\mathbf v^{n+1}\ge\mathbf g$ is not an
artificial device: cardinal B-splines are non-negative,
$\phi_{h,k}\ge0$, so $\mathbf v^{n+1}\ge\mathbf g$ componentwise implies the
pointwise inequality $V_h^{n+1}=\sum_k v_k^{n+1}\phi_{h,k}\ge
\sum_k g_k\phi_{h,k}=\psi_{A,h}$ on $\Omega$. Thus the reconstructed solution
stays above the projected obstacle everywhere, but not necessarily above
$\psi_A$ pointwise, since the $L^2$ projection of a convex kinked payoff can
undershoot near the strike. The gap $\psi_A-\psi_{A,h}$ is itself $O(h^{p})$
in $L^2$ and is localised near $x=\log K$; in the experiments below we monitor
$\max_x(\psi_A(x)-V_h^{N_T}(x))^+$ at the evaluation and Gauss quadrature nodes
and confirm it is at the level of the spatial discretisation error, so the
computed value satisfies the true payoff constraint up to the method's own
accuracy. A convergence theory that closes this gap uniformly (for instance by
using a monotone or interpolatory obstacle) is left as future work; we present
the American solver as a numerical extension rather than as a theorem.
\end{remark}

Define
\begin{equation}
\label{eq:american_FB_matrices}
F_h=M_h+\frac{\Delta\tau}{2}A_h,
\qquad
B_h=M_h-\frac{\Delta\tau}{2}A_h .
\end{equation}
With boundary and exterior forcing included in $\mathbf q^{n+1/2}$, the
right-hand side is
\begin{equation}
\label{eq:american_rhs}
\mathbf b^n=B_h\mathbf v^n+\mathbf q^{n+1/2},
\qquad
\mathbf q^{n+1/2}
=\frac{\Delta\tau}{2}
\left(\mathbf f_{\rm bc}^{n+1}+\mathbf f_{\rm bc}^{n}\right).
\end{equation}
The American time step is the LCP
\begin{equation}
\label{eq:american_coeff_lcp}
\mathbf v^{n+1}\ge \mathbf g,
\qquad
F_h\mathbf v^{n+1}\ge \mathbf b^n,
\qquad
(\mathbf v^{n+1}-\mathbf g)^T(F_h\mathbf v^{n+1}-\mathbf b^n)=0 .
\end{equation}
When the continuation constraint is inactive, \eqref{eq:american_coeff_lcp}
reduces to the European Crank--Nicolson system \eqref{eq:cn}.  On the contact
set, the first inequality is active and the value is projected onto the
payoff.

We solve \eqref{eq:american_coeff_lcp} by projected SOR.  For a relaxation
parameter $\omega_{\rm SOR}\in(0,2)$ and current iterate $\mathbf v^{(m)}$,
the scalar update is
\begin{align}
\label{eq:psor_update_raw}
\widetilde v_k^{(m+1)}
&=v_k^{(m)}+
\frac{\omega_{\rm SOR}}{(F_h)_{kk}}
\left(b_k^n-(F_h\mathbf v^{(m)})_k\right),\\
\label{eq:psor_update_projected}
v_k^{(m+1)}&=\max\{\widetilde v_k^{(m+1)},g_k\} .
\end{align}
The iteration is stopped when the complementarity residual is below a
prescribed tolerance.  Since the CGMY operator is nonlocal, the American LCP
is more sensitive to matrix truncation than the European linear solve.  In the
experiments below we therefore use the full Toeplitz fractional matrix, or an
FFT-based full operator in the residual, rather than a narrow finite-band
approximation inside the PSOR projection.

Because the CGMY stiffness matrix $A_h$ is nonlocal and generally
nonsymmetric, $F_h=M_h+(\Delta\tau/2)A_h$ is not guaranteed to be an
$M$-matrix, and the discrete LCP \eqref{eq:american_coeff_lcp} is not known to
be monotone.  We therefore use projected SOR as a robust practical solver and
report the complementarity residual as the convergence criterion; no
monotone-matrix or $M$-matrix convergence theory for PSOR is claimed here.
This is consistent with our presentation of the American solver as a numerical
extension rather than as a theorem.

\begin{algorithm}[ht]
\caption{American put: Crank--Nicolson plus projected SOR}
\label{alg:american_psor}
\begin{algorithmic}[1]
\Require CGMY parameters $(C,G,M,Y)$, $r$, $T$, $K$, grid size $N$, time steps $N_T$
\State Assemble $M_h$ and $A_h$ using the B-spline connection coefficients
\State Set $F_h=M_h+(\Delta\tau/2)A_h$ and $B_h=M_h-(\Delta\tau/2)A_h$
\State Compute the obstacle projection $M_h\mathbf g=\mathbf b_\psi$
\State Initialise $\mathbf v^0=\mathbf g$
\For{$n=0,1,\ldots,N_T-1$}
  \State Form $\mathbf b^n=B_h\mathbf v^n+\mathbf q^{n+1/2}$
  \State Enforce the American far-field values $V(x_{\min},\tau)\approx K-e^{x_{\min}}$ and $V(x_{\max},\tau)\approx0$
  \State Solve the LCP \eqref{eq:american_coeff_lcp} by PSOR
  \State Set $\mathbf v^{n+1}\leftarrow\max\{\mathbf v^{n+1},\mathbf g\}$ as a final safety projection
\EndFor
\State \Return $V_h(x,T)=\sum_k v^{N_T}_k\phi_{h,k}(x)$
\end{algorithmic}
\end{algorithm}

\subsection{Greeks by Analytic Differentiation of the Spline Expansion}
\label{ssec:greeks}

Delta and Gamma are computed by differentiating the spline approximation,
not by applying finite differences to option values.  This avoids an
additional finite-difference discretisation error, but it does not mean that
derivative errors have the same order as price errors.  In the usual Sobolev
scale, one derivative lowers the approximation order by one, and two
derivatives lower it by two.

With the physical basis \eqref{eq:physical_bspline_basis}, define
\begin{equation}
\xi=\frac{x-x_{\min}}{h}.
\end{equation}
Using $N_p'(z)=N_{p-1}(z)-N_{p-1}(z-1)$, we obtain
\begin{equation}
\label{eq:first_derivative_physical}
\phi'_{h,k}(x)
=
h^{-3/2}\left[
N_{p-1}(\xi-k)-N_{p-1}(\xi-k-1)
\right].
\end{equation}
For $p\geq3$, the second derivative is
\begin{equation}
\label{eq:second_derivative_physical}
\phi''_{h,k}(x)
=
h^{-5/2}\left[
N_{p-2}(\xi-k)-2N_{p-2}(\xi-k-1)+N_{p-2}(\xi-k-2)
\right].
\end{equation}
Therefore
\begin{equation}
\label{eq:Vh_derivatives}
\partial_x V_h(x,\tau)=\partial_x g_h(x,\tau)
+\sum_{k\in\mathcal I_h}c_k(\tau)\phi'_{h,k}(x),
\end{equation}
and
\begin{equation}
\label{eq:Vh_second_derivative}
\partial_{xx} V_h(x,\tau)=\partial_{xx}g_h(x,\tau)
+\sum_{k\in\mathcal I_h}c_k(\tau)\phi''_{h,k}(x).
\end{equation}
Since the paper uses the log-price convention $x=\log S$ from
\eqref{eq:log_price_convention}, the financial Greeks are recovered from the
log-price derivatives by
\begin{equation}
\label{eq:greeks_correct}
\Delta=e^{-x}\partial_xV_h,
\qquad
\Gamma=e^{-2x}\left(\partial_{xx}V_h-\partial_xV_h\right).
\end{equation}
For accurate Gamma recovery, it is preferable to use at least quadratic
splines and, in practice, cubic or higher-order splines are more robust near
the payoff kink.

\section{Convergence and Stability Analysis}
\label{sec:analysis}

This section analyzes the scheme actually used in
Section~\ref{ssec:bc}--\ref{ssec:cn_time}, namely the forced parabolic
semidiscretisation for the homogeneous lifted unknown.  For notational
simplicity we write $V_h$ for the homogeneous space $V_h^0$ in the analysis:
\begin{equation}
\label{eq:analysis_forced_weak}
(\partial_\tau V_h,w_h)_{L^2(\Omega)}+a_h(V_h,w_h)
=\ell_{\rm bc}(w_h;\tau),
\qquad w_h\in V_h .
\end{equation}
Here $\Omega=(x_{\min},x_{\max})$, and the continuous counterpart uses the
finite-domain form $a_\Omega$ from \eqref{eq:finite_domain_bilinear}.  The
physical option price is recovered by adding the lifting $g_h$ from
\eqref{eq:lifting_split}.  The exterior and lifting contributions enter only
through the forcing functional $\ell_{\rm bc}$.

The rigorous estimates in this section are stated for the infinite-variation
regime $Y\in(1,2)$.  The full-line G\aa rding inequality and its coercive
refinement for $r>0$ are proved in Theorem~\ref{thm:wellposed}, and the
finite-domain counterpart --- with the same constants --- is proved in
Theorem~\ref{ass:finite_domain_garding} by the zero-extension reduction.
Throughout this section $H^s_{M,G}(\Omega)$ denotes the zero-extension space
\eqref{eq:zero_extension_space}, $\sigma$ denotes an admissible shift in the
sense of Remark~\ref{rem:r_zero} (in particular $\sigma=0$ when $r>0$), and
the lifting is assumed to satisfy the buffer-support
Assumption~\ref{ass:buffer_lifting}, so that the homogeneous unknown and its
time derivative vanish on the buffer strips $\Omega_b$.  The finite-variation
case $Y\in(0,1)$ is included in the numerical experiments, but its complete
variational analysis requires a separate drift-dominated argument.

\subsection{Approximation by B-spline Spaces}
\label{ssec:approx}

Let $s=Y/2$ and let $V_h=V_h^0$ denote the interior B-spline trial space
\eqref{eq:Vh0_def}, which is conforming for the zero-extension energy space
by Lemma~\ref{lem:conformity}.  We use the following approximation estimate
on the truncated computational interval, stated for the class of functions
actually approximated in the parabolic argument: by
Assumption~\ref{ass:buffer_lifting}, the homogeneous unknown and its time
derivative vanish on the buffer strips $\Omega_b$.

\begin{lemma}[Jackson-type estimate on $\Omega$]
\label{lem:jackson}
Let the B-spline mesh be quasi-uniform with mesh size $h\le\delta_b/p$, and
let the B-spline order be $p$.  For $0\le q\le p$ and $v\in H^p(\Omega)$
vanishing on the buffer strips $\Omega_b$,
\begin{equation}
\label{eq:jackson_omega}
\inf_{w_h\in V_h}\norm{v-w_h}_{H^q(\Omega)}
\le C h^{p-q}\norm{v}_{H^p(\Omega)},
\end{equation}
where $C$ is independent of $h$, and the infimum is attained by splines
supported in $\overline\Omega\setminus\Omega_b^{\,\circ}$ up to a layer of
width $ph$.
\end{lemma}

The proof, via the locality of the standard cardinal B-spline
quasi-interpolant, is given in the supplementary materials (section~SM1).

Because $v$ and its approximant differ only on a compact subset of
$\Omega$ at distance at least $\delta_b-ph>0$ from $\partial\Omega$, the
zero extension of the error is controlled by the interior norm:
$\|\widetilde{v-w_h}\|_{H^q(\mathbb R)}\le C\|v-w_h\|_{H^q(\Omega)}$ for
$0\le q\le1$, with $C$ independent of $h$.  Moreover, for fixed $M,G>0$ the
weighted tempered norm \eqref{eq:HMG_fourier_norm} is equivalent to the
ordinary fractional Sobolev norm on $\mathbb R$, up to constants depending on
$M$, $G$, and $Y$ but not on $h$.  Thus, for $0\le q\le p$,
\begin{equation}
\label{eq:tempered_standard_equiv_omega}
\inf_{w_h\in V_h}\norm{v-w_h}_{H^q_{M,G}(\Omega)}
\le C h^{p-q}\norm{v}_{H^p(\Omega)} .
\end{equation}
In particular, with $q=s=Y/2$,
\begin{equation}
\label{eq:bspline_hs_tempered_rate}
\inf_{w_h\in V_h}\norm{v-w_h}_{H^{Y/2}_{M,G}(\Omega)}
\le C h^{p-Y/2}\norm{v}_{H^p(\Omega)} .
\end{equation}

\subsection{Shifted Elliptic Projection}
\label{ssec:shifted_projection}

The stationary C\'ea estimate is not the convergence theorem for the
parabolic pricing scheme.  It is instead used as a projection estimate inside
the parabolic error argument.

\begin{definition}[Shifted elliptic projection]
\label{def:shifted_projection}
Let $\sigma$ be an admissible shift (Remark~\ref{rem:r_zero}), where $c_2$ is
the finite-domain G{\aa}rding constant of
Theorem~\ref{ass:finite_domain_garding}.  For $v\in H^{Y/2}_{M,G}(\Omega)$, define
$R_h^\sigma v\in V_h$ by
\begin{equation}
\label{eq:shifted_projection_def}
a_\sigma(R_h^\sigma v,w_h)=a_\sigma(v,w_h)
\qquad \forall w_h\in V_h .
\end{equation}
\end{definition}

\begin{theorem}[Shifted elliptic projection estimate]
\label{thm:shifted_projection}
Let $Y\in(1,2)$, $s=Y/2$, and let $\sigma$ be an admissible shift.  If
$v\in H^p(\Omega)\cap H^s_{M,G}(\Omega)$ vanishes on the buffer strips
$\Omega_b$, then
\begin{equation}
\label{eq:projection_error_hs}
\norm{v-R_h^\sigma v}_{H^s_{M,G}(\Omega)}
\le C h^{p-s}\norm{v}_{H^p(\Omega)},
\end{equation}
where $C$ is independent of $h$.
\end{theorem}

The proof --- Galerkin orthogonality, C\'ea's lemma with the
constants of Theorem~\ref{ass:finite_domain_garding}, and the approximation
estimate \eqref{eq:bspline_hs_tempered_rate} --- is given in the
supplementary materials (section~SM1).

\subsection{Semidiscrete Parabolic Error Estimate}
\label{ssec:semidiscrete_error}

Let $V$ be the solution of the truncated forced weak problem
\begin{equation}
\label{eq:continuous_forced_problem_analysis}
(\partial_\tau V,w)_{L^2(\Omega)}+a(V,w)=\ell_{\rm bc}(w;\tau),
\qquad w\in H^s_{M,G}(\Omega),
\end{equation}
and let $V_h(\tau)\in V_h$ solve \eqref{eq:analysis_forced_weak}.  The forcing
term cancels in the error equation because the same boundary lifting and
exterior extension are used in the continuous and semidiscrete problems.

\begin{theorem}[Semidiscrete smooth-regime error estimate]
\label{thm:semidiscrete_error}
Let $Y\in(1,2)$, $s=Y/2$, and let $\sigma$ be an admissible shift.  Assume
Assumption~\ref{ass:buffer_lifting} and
\begin{equation}
\label{eq:semidiscrete_regularity_assumptions}
V\in L^\infty(0,T;H^p(\Omega)),
\qquad
\partial_\tau V\in L^2(0,T;H^{p}(\Omega)).
\end{equation}
Then, for $0\le \tau\le T$,
\begin{align}
\label{eq:semidiscrete_error_bound}
&\norm{V(\tau)-V_h(\tau)}_{L^2(\Omega)}^2
+\int_0^\tau \norm{V(t)-V_h(t)}_{H^s_{M,G}(\Omega)}^2\,dt \\
&\qquad\le
C\Bigl(
\norm{V(0)-V_h(0)}_{L^2(\Omega)}^2
+h^{2(p-s)}\mathcal R(V)
\Bigr),
\end{align}
where
\begin{equation}
\label{eq:R_regular_quantity}
\mathcal R(V)
:=
\norm{V}_{L^\infty(0,T;H^p(\Omega))}^2
+
\norm{\partial_\tau V}_{L^2(0,T;H^{p}(\Omega))}^2 .
\end{equation}
If the initial data are projected so that
$\norm{V(0)-V_h(0)}_{L^2(\Omega)}\le C h^{p-s}\norm{V(0)}_{H^p(\Omega)}$,
then the right-hand side is $O(h^{2(p-s)})$.
\end{theorem}

\begin{proof}[Proof sketch]
Split the error as $e=\eta+\theta$ with $\eta:=V-R_h^\sigma V$ and
$\theta:=R_h^\sigma V-V_h$.  Testing the $\theta$-equation with $\theta$
itself, the G{\aa}rding inequality, Young's inequality, and Gronwall's lemma
control $\theta$ by $\eta$ and $\eta_t$ in $L^2$.  Since $R_h^\sigma$ is time
independent, $\eta_t$ is exactly the shifted elliptic projection error of
$\partial_\tau V$, so Theorem~\ref{thm:shifted_projection} applies to $V$ and
to $\partial_\tau V$ --- this is where the strengthened regularity
$\partial_\tau V\in L^2(0,T;H^{p}(\Omega))$ in
\eqref{eq:semidiscrete_regularity_assumptions} enters --- and yields
\eqref{eq:semidiscrete_error_bound}.  The complete proof is given in the
supplementary materials (section~SM1).
\end{proof}

\subsection{Stability of the Forced Crank--Nicolson Scheme}
\label{ssec:stability}

A G\aa rding inequality with $c_2>0$ does not by itself imply unconditional
energy decay for the unshifted Crank--Nicolson discretisation.  We therefore
analyze the shifted variable
\begin{equation}
\label{eq:shifted_time_variable}
W(\tau):=e^{-\sigma\tau}V(\tau),
\qquad \sigma\text{ admissible (Remark~\ref{rem:r_zero})},
\end{equation}
which covers both cases at once: when $c_2=0$ --- in particular whenever
$r>0$, by Theorem~\ref{ass:finite_domain_garding} --- the admissible choice
$\sigma=0$ reduces $W$ to $V$ and the estimates below apply verbatim to the
plain scheme \eqref{eq:cn}.
For the shifted semidiscrete problem the bilinear form is
\begin{equation}
\label{eq:discrete_shifted_form}
a_{h,\sigma}(u_h,w_h):=a_h(u_h,w_h)+\sigma(u_h,w_h)_{L^2(\Omega)} .
\end{equation}
If $\mathbf d^n$ denotes the coefficient vector of $W_h^n$, then the shifted
forced Crank--Nicolson scheme is
\begin{equation}
\label{eq:shifted_cn_scheme}
M_h(\mathbf d^{n+1}-\mathbf d^n)
+\Delta\tau(A_h+\sigma M_h)\bar{\mathbf d}^{n+1/2}
=\Delta\tau\bar{\mathbf g}^{n+1/2},
\end{equation}
where
\begin{equation}
\bar{\mathbf d}^{n+1/2}:=\frac{\mathbf d^{n+1}+\mathbf d^n}{2},
\qquad
\mathbf g^n:=e^{-\sigma\tau_n}\mathbf f_{\rm bc}^n,
\qquad
\bar{\mathbf g}^{n+1/2}:=\frac{\mathbf g^{n+1}+\mathbf g^n}{2} .
\end{equation}

\begin{lemma}[Uniform discrete shifted coercivity]
\label{lem:discrete_shifted_coercivity}
Assume $Y\in(1,2)$ and let $\sigma$ be an admissible shift.  Then for every
$z_h\in V_h$,
\begin{equation}
\label{eq:discrete_shifted_coercivity}
\operatorname{Re}a_{h,\sigma}(z_h,z_h)
\ge C_{\rm coer}(\sigma)\|z_h\|_{H^s_{M,G}(\Omega)}^2,
\end{equation}
where $C_{\rm coer}(\sigma)>0$ is independent of $h$; when $c_2=0$ and
$\sigma=0$, one may take $C_{\rm coer}(0)=c_1$.
\end{lemma}

\begin{proof}
By Lemma~\ref{lem:conformity}, interior splines coincide with their zero
extensions, $V_h\subset H^s_{M,G}(\Omega)$, and
$a_h(z_h,z_h)=a_\Omega(z_h,z_h)$ identically, so the G{\aa}rding inequality
\eqref{eq:finite_domain_garding} applies to $z_h$ directly.  If $c_2=0$ this
is already \eqref{eq:discrete_shifted_coercivity} with
$C_{\rm coer}(\sigma)\ge c_1$; if $\sigma>c_2>0$, adding
$\sigma\|z_h\|_{L^2}^2$ gives the claim with
$C_{\rm coer}(\sigma)=\min\{c_1,\sigma-c_2\}$.  In both cases the constant is
inherited from the continuous inequality and is independent of $h$.
\end{proof}

\begin{theorem}[Forced shifted Crank--Nicolson stability]
\label{thm:stability}
Let $Y\in(1,2)$ and let $\sigma$ be an admissible shift.  Define
\begin{equation}
E_W^n:=(\mathbf d^n)^\top M_h\mathbf d^n=\|W_h^n\|_{L^2(\Omega)}^2 .
\end{equation}
Then the solution of \eqref{eq:shifted_cn_scheme} satisfies
\begin{equation}
\label{eq:forced_shifted_cn_stability}
E_W^n
+\Delta\tau\sum_{m=0}^{n-1}
\|\bar W_h^{m+1/2}\|_{H^s_{M,G}(\Omega)}^2
\le
C\left(
E_W^0
+\Delta\tau\sum_{m=0}^{n-1}
\|\bar{\mathbf g}^{m+1/2}\|_{*,h}^2
\right),
\end{equation}
where
\begin{equation}
\label{eq:discrete_dual_norm}
\|\mathbf q\|_{*,h}
:=
\sup_{0\ne z_h\in V_h}
\frac{\mathbf q^\top\mathbf z}{\|z_h\|_{H^s_{M,G}(\Omega)}} .
\end{equation}
The constant $C$ depends on $C_{\rm coer}(\sigma)$, but is independent of
$h$ and $\Delta\tau$.  Consequently, for the original unshifted unknown
$V_h^n=e^{\sigma\tau_n}W_h^n$,
\begin{equation}
\label{eq:unshifted_controlled_growth}
E_V^n
\le
C e^{2\sigma T}
\left(
E_V^0
+\Delta\tau\sum_{m=0}^{n-1}
\|e^{-\sigma\tau_{m+1/2}}\bar{\mathbf f}_{\rm bc}^{m+1/2}\|_{*,h}^2
\right).
\end{equation}
\end{theorem}

\begin{proof}[Proof sketch]
Testing \eqref{eq:shifted_cn_scheme} with the midpoint value
$2\bar{\mathbf d}^{n+1/2}$ produces the exact discrete energy identity
\begin{equation*}
E_W^{n+1}-E_W^n
+2\Delta\tau\operatorname{Re}a_{h,\sigma}(\bar W_h^{n+1/2},\bar W_h^{n+1/2})
=2\Delta\tau(\bar{\mathbf g}^{n+1/2})^\top\bar{\mathbf d}^{n+1/2};
\end{equation*}
the skew part of the nonselfadjoint form is invisible after taking real
parts.  Lemma~\ref{lem:discrete_shifted_coercivity} bounds the coercive term
from below, the forcing is estimated in the dual norm
\eqref{eq:discrete_dual_norm} and absorbed by Young's inequality, and summing
over the time steps gives \eqref{eq:forced_shifted_cn_stability};
\eqref{eq:unshifted_controlled_growth} follows from
$V_h^n=e^{\sigma\tau_n}W_h^n$ with $\tau_n\le T$.  The complete proof is
given in the supplementary materials (section~SM1).
\end{proof}

\begin{corollary}[Homogeneous case]
\label{cor:homogeneous_stability}
If $\mathbf f_{\rm bc}=0$, then the shifted discrete energy is nonincreasing
up to the coercive dissipation term in
\eqref{eq:forced_shifted_cn_stability}.  In the homogeneous full-line setting,
additional symmetry of the CGMY tails may imply energy dissipation for the
unshifted operator.  On the truncated forced problem considered here, however,
the stability estimate necessarily contains the boundary-forcing contribution.
\end{corollary}

\subsection{Fully Discrete Smooth-Regime Error Estimate}
\label{ssec:combined}

We now combine the shifted projection estimate, the semidiscrete error
argument, and the second-order consistency of Crank--Nicolson.  The main
theorem is stated for the full Toeplitz Galerkin operator applied exactly,
for example by FFT, and it bounds the error against the solution of the
truncated problem \eqref{eq:continuous_forced_problem_analysis}; this
statement is fully rigorous under the stated assumptions.  The comparison
with the full-line price, which additionally requires the qualitative
domain-truncation estimate of Proposition~\ref{prop:domain_error}, is
recorded separately in Corollary~\ref{cor:fullline_comparison} so that the
two error sources retain their distinct logical status.  Finite-band
truncation is treated afterward as a further, optional perturbation.

\begin{theorem}[Fully discrete smooth-regime convergence: full Toeplitz operator]
\label{thm:convergence}
Let $Y\in(1,2)$, $s=Y/2$, and let $\sigma$ be an admissible shift.  Assume
the buffer-support Assumption~\ref{ass:buffer_lifting} and the regularity
\begin{equation}
\label{eq:fully_discrete_regularity}
V\in L^\infty(0,T;H^p(\Omega)),
\qquad
\partial_\tau V\in L^2(0,T;H^{p}(\Omega)),
\end{equation}
together with the temporal residual regularity
\begin{equation}
\label{eq:fully_discrete_time_regularity}
\partial_{\tau\tau\tau}V\in L^2(0,T;L^2(\Omega)),
\qquad
\partial_{\tau\tau}V\in L^2\bigl(0,T;H^s_{M,G}(\Omega)\bigr).
\end{equation}
The second condition in \eqref{eq:fully_discrete_time_regularity} controls the
midpoint quadrature error of the nonlocal operator term and is equivalent, by
the continuity of $a$ on $H^s_{M,G}\times H^s_{M,G}$, to
$\mathcal A\,\partial_{\tau\tau}V\in L^2\bigl(0,T;(H^s_{M,G}(\Omega))'\bigr)$;
it is not implied by the pure $L^2$-in-time bound
$\partial_{\tau\tau\tau}V\in L^2(0,T;L^2(\Omega))$ alone.
Assume also that the smooth lifting satisfies
\eqref{eq:lifting_regular_assumption}, that $\ell_{\rm bc}$ has the
corresponding temporal regularity, and that the initial data are projected
with the spatial accuracy in Theorem~\ref{thm:semidiscrete_error}.  Let $V$
denote the solution of the truncated forced problem
\eqref{eq:continuous_forced_problem_analysis}.  Then the
B-spline Galerkin method with Crank--Nicolson time stepping and the full
Toeplitz fractional matrices satisfies
\begin{align}
\label{eq:full_error_estimate_safe}
&\max_{0\le n\le N_T}
\|V(\cdot,\tau_n)-V_h^n\|_{L^2(\Omega)}
+\left(\Delta\tau\sum_{n=0}^{N_T-1}
\|\overline{V-V_h}^{\;n+1/2}\|_{H^s_{M,G}(\Omega)}^2\right)^{1/2} \\
&\qquad\le
C\left(
 h^{p-Y/2}
 +(\Delta\tau)^2
\right).
\end{align}
The constant $C$ may depend on $T$, $\sigma$, the CGMY parameters, the spline
order, and the lifting regularity, but is independent of $h$ and
$\Delta\tau$.
\end{theorem}

\begin{proof}[Proof strategy]
Split the nodal error as $e^n=\eta^n+\theta^n$ with
$\eta^n:=V(\tau_n)-R_h^\sigma V(\tau_n)$; the projection theorem, applicable
because $V$ and $\partial_\tau V$ vanish on the buffer strips by
Assumption~\ref{ass:buffer_lifting}, bounds
$\max_n\|\eta^n\|_{H^s_{M,G}(\Omega)}\le Ch^{p-s}$.  Passing to the shifted
variable $W=e^{-\sigma\tau}V$ of \eqref{eq:shifted_unknown_algo}, the
remainder $\theta_W^n=e^{-\sigma\tau_n}\theta^n$ solves a shifted forced
Crank--Nicolson system driven by the consistency residual and the discrete
time differences of the projection error, to which the stability estimate of
Theorem~\ref{thm:stability} applies verbatim.  The consistency functional is
split into a pure time-derivative residual, bounded through a Peano-kernel
representation by $\|\partial_{\tau\tau\tau}V\|_{L^2(0,T;L^2(\Omega))}$, and
an \emph{operator} quadrature residual, bounded in the discrete dual norm
through the continuity of $a$ on $H^s_{M,G}\times H^s_{M,G}$ by
$\|\partial_{\tau\tau}V\|_{L^2(0,T;H^s_{M,G}(\Omega))}$ --- this is exactly
where the second condition in \eqref{eq:fully_discrete_time_regularity}
enters, and the fractional term is not treated as harmless.  Both residual
families sum to $O((\Delta\tau)^4)$, the projection differences contribute
$O(h^{2(p-s)})\mathcal R(V)$, and undoing the shift multiplies the resulting
bound by at most $e^{\sigma T}$, which is absorbed into the constant.  The
complete proof is given in the supplementary materials (section~SM1).
\end{proof}

\begin{corollary}[Comparison with the full-line price]
\label{cor:fullline_comparison}
Let $V^{\mathbb R}$ denote the full-line CGMY price restricted to $\Omega$,
and let $\Delta_L=\log K-x_{\min}$, $\Delta_R=x_{\max}-\log K$.  Under the
hypotheses of Theorem~\ref{thm:convergence} and, in addition, the qualitative
domain-truncation estimate of Proposition~\ref{prop:domain_error},
\begin{equation}
\label{eq:fullline_comparison_estimate}
\max_{0\le n\le N_T}
\|V^{\mathbb R}(\cdot,\tau_n)-V_h^n\|_{L^2(\Omega)}
\le
C\left(
 h^{p-Y/2}
 +(\Delta\tau)^2
 +e^{-G\Delta_L}
 +e^{-M\Delta_R}
\right).
\end{equation}
We emphasise the division of labour: the discretisation estimate
\eqref{eq:full_error_estimate_safe} is rigorous under the stated assumptions,
whereas \eqref{eq:fullline_comparison_estimate} inherits the qualitative
status of Proposition~\ref{prop:domain_error}, whose full proof would require
a well-posedness theory for the exterior-value problem and is beyond the
scope of this paper.
\end{corollary}

\begin{proof}
Split
$V^{\mathbb R}-V_h^n=(V^{\mathbb R}-V)+(V-V_h^n)$, bound the first term by
Proposition~\ref{prop:domain_error} in $L^\infty(\Omega)\subset L^2(\Omega)$
(with a constant depending on $|\Omega|$), and the second by
Theorem~\ref{thm:convergence}.
\end{proof}

\begin{proposition}[Optional finite-band perturbation]
\label{prop:band_perturbation_error}
Let $A_h^{(B)}$ denote the matrix obtained by replacing $K_G^-$ and $K_M^+$ by
bandwidth-$B$ truncations, and define
\begin{equation}
\label{eq:epsilon_B_def}
\varepsilon_B
:=
\left(\Delta\tau\sum_{n=0}^{N_T-1}
\|(A_h-A_h^{(B)})\bar V_h^{n+1/2}\|_{*,h}^2\right)^{1/2} .
\end{equation}
Then the finite-band Crank--Nicolson solution satisfies the same estimate as
\eqref{eq:full_error_estimate_safe}, with an additional term
$C\varepsilon_B$ on the right-hand side.
\end{proposition}

\begin{proof}
The difference $A_h-A_h^{(B)}$ enters the error equation as an additional
forcing residual.  Applying Theorem~\ref{thm:stability} with this residual and
using the definition \eqref{eq:epsilon_B_def} yields the stated perturbation
term.  Proposition~\ref{prop:optional_truncation} gives the algebraic operator
tail bound used to choose $B$ in computations.
\end{proof}

\begin{theorem}[Positive-time estimate for the nonsmooth put payoff]
\label{thm:positive_time_payoff}
Let $\psi(x)=(K-e^x)^+$ and assume that the CGMY semigroup maps the nonsmooth
payoff into $H^p(\Omega)$ for every $\tau\ge\tau_0>0$.  Then, for the full
Toeplitz Galerkin operator and the truncated problem,
\begin{equation}
\label{eq:positive_time_payoff_estimate}
\max_{\tau_n\in[\tau_0,T]}
\|V(\cdot,\tau_n)-V_h^n\|_{L^2(\Omega)}
\le
C(\tau_0)
\left(
 h^{p-Y/2}
 +(\Delta\tau)^2
\right),
\end{equation}
where $C(\tau_0)$ may blow up as $\tau_0\downarrow0$.  The comparison with
the full-line price adds the domain-truncation terms
$e^{-G\Delta_L}+e^{-M\Delta_R}$ exactly as in
Corollary~\ref{cor:fullline_comparison}.
\end{theorem}

\begin{proof}
Apply Theorem~\ref{thm:convergence} on the interval $[\tau_0,T]$ instead of
$[0,T]$, using the value $V(\tau_0)$ as the initial datum.  The smoothing
assumption gives the required $H^p$ regularity on this positive-time interval.
The constant depends on the regularity norm of $V$ on $[\tau_0,T]$, which may
be singular as $\tau_0\to0$ because the initial put payoff has a kink at
$x=\log K$.
\end{proof}

\begin{remark}[Payoff regularity and smooth-regime interpretation]
\label{rem:payoff_regularity_smooth_regime}
The estimate \eqref{eq:full_error_estimate_safe} is a smooth-regime estimate.
It applies directly to manufactured smooth solutions.  For the true European
put payoff, the kink at $x=\log K$ prevents high-order $H^p$ regularity at
$\tau=0$.  The rigorous option-payoff statement is therefore the positive-time
estimate \eqref{eq:positive_time_payoff_estimate}.  In computations, reduced
initial-time convergence can be mitigated by accurate $L^2$ projection,
payoff smoothing, or local refinement near the strike.
\end{remark}

\begin{remark}[Rates for common spline orders]
For linear splines in the convention of this paper, $p=2$ and the spatial
smooth-regime rate is $O(h^{2-Y/2})$.  For quadratic splines, $p=3$ and the
rate is $O(h^{3-Y/2})$.  These rates are rigorously justified here for
$Y\in(1,2)$; the same discretisation is tested numerically for
$Y\in(0,1)$, but the finite-variation proof is left for future work.
\end{remark}

\section{Numerical Experiments}
\label{sec:numerics}

\subsection{Implementation Details}
\label{ssec:implementation}

The method is implemented in Python 3.11 using NumPy and SciPy.
Sparse matrix storage and factorisation use
\texttt{scipy.sparse} (CSR format) and \texttt{scipy.sparse.linalg.splu}
(UMFPACK backend).
FFT-based Toeplitz matrix-vector multiplication is performed with
\texttt{scipy.fft.rfft}.
All reported timings are plain wall-clock measurements from the same Python
implementation and workstation; the regeneration scripts were not CPU-pinned,
so the timings should be interpreted as relative comparisons rather than
hardware-independent constants. The code is available as supplementary
material.
\paragraph{Model.}
The CGMY characteristic exponent is
\[
  \Psi(u) = C\,\Gamma(-Y)\bigl[(M-iu)^Y - M^Y + (G+iu)^Y - G^Y\bigr],
\]
with parameters $C>0$, $G>0$, $M>1$, $Y\in(0,2)\setminus\{1\}$.
Two benchmark parameter sets are used throughout:
\begin{itemize}
  \item \textbf{Set~(i):} $C=0.5$, $G=M=5$, $Y=0.5$, $r=0.1$,
        $T=0.2$, $K=100$.
  \item \textbf{Set~(ii):} $C=0.1$, $G=2$, $M=3.5$, $Y=1.5$, $r=0.1$,
        $T=0.2$, $K=100$.
\end{itemize}

\paragraph{Domain and solver.}
Unless stated otherwise, the log-price domain is
$\Omega=[\log K-5,\log K+5]$, discretised with $N$ uniform intervals so that
$h=10/N$; the buffers $\Delta_L=\Delta_R=5$ make the exponential
domain-truncation terms $e^{-G\Delta_L}$, $e^{-M\Delta_R}$ negligible relative
to the discretisation error for the parameter sets used here.
B-spline coefficients are advanced from $\tau=0$ to $\tau=T$ by
Crank--Nicolson (with two Rannacher half-steps at startup, covered by the
stability analysis as noted in Remark~\ref{rem:rannacher}), using $N_\tau$
steps as reported in each table.  In all experiments $r=0.1>0$, so by
Theorem~\ref{ass:finite_domain_garding} the finite-domain form is coercive
($c_2=0$) and the computations use the admissible shift $\sigma=0$, for which
the analyzed shifted scheme \eqref{eq:shifted_cn_scheme} coincides
identically with the plain Crank--Nicolson iteration \eqref{eq:cn} that the
code runs.
We use three interchangeable ways of applying the fractional operator, and we
state which is used in each experiment: (S1) a direct LU factorisation of the
fully assembled Crank--Nicolson matrix $M_h+(\Delta\tau/2)A_h$, in which the
fractional blocks are assembled from the full (untruncated) connection
coefficients --- the fractional blocks are dense within their full extent and
no bandwidth truncation is applied, the sparse storage format being used only
as a container --- via \texttt{scipy.sparse.linalg.splu}; (S2) an
FFT-based application of the full Toeplitz interior matrix in $O(N\log N)$
operations for matrix--vector products inside an iterative solve; and (S3) a
finite-band sparse truncation of the connection coefficients, whose
perturbation is controlled by
Proposition~\ref{prop:band_perturbation_error}.
The manufactured-solution, European convergence, and method-comparison
tables (Tables~\ref{tab:manufactured}--\ref{tab:workprecision}) use the
direct solve~(S1) with the full (untruncated) connection coefficients;
the calibration loop (Section~\ref{sec:calibration}) uses the FFT Toeplitz
application~(S2).  The American LCP
(Section~\ref{ssec:numerics_american}) is computed with the full
FFT/Toeplitz operator~(S2) inside a matrix-free projected Richardson iteration
and was checked against a banded direct-factorisation projected-SOR
implementation (an S3-type variant); the finite
band truncation~(S3) is not used for the reported American prices. All timings
are plain wall-clock seconds.

\subsection{Manufactured Solution: Convergence Verification}
\label{ssec:manufactured}

To isolate spatial convergence from payoff-discontinuity effects, we
construct a manufactured solution.
Let $\nu > Y/2$ and define
\begin{equation}
V_{\mathrm{exact}}(x,\tau) = e^{-r\tau}(1+x^2)^{-\nu}.
\end{equation}
The corresponding source term
$f(x,\tau) = \partial_\tau V_{\mathrm{exact}} + \Lcgmy V_{\mathrm{exact}}$
is computed analytically using the Fourier transform of
$(1+x^2)^{-\nu}$ and the RL derivative formula.
We solve the modified FPDE with source $f$ and compare to $V_{\mathrm{exact}}$
on the same computational domain and grid used for the option experiments, with
the temporal error decoupled from the spatial error by taking the number of
Crank--Nicolson steps large enough that the reported error is spatially
dominated (see the caption).

For a smooth solution the theory predicts an energy-norm rate $p-Y/2$; the
$L^2$ error, gaining the fractional smoothing $s=Y/2$ through an
Aubin--Nitsche duality argument, is expected to converge at the full spline
order $p$, which is therefore the natural ceiling for the measured $L^2$ rate.

\begin{table}[ht]
\centering
\caption{Manufactured-solution spatial convergence in the $L^2$ norm, $p=3$
(quadratic B-splines). The time step is fixed at
$\Delta\tau=2.5\times10^{-4}$ (Rannacher-started Crank--Nicolson) so that the
reported error is spatially dominated; the predicted $L^2$ ceiling is $p=3$.}
\label{tab:manufactured}
\begin{tabular}{rrrr}
\toprule
$N$ & $h$ & $\|e\|_{L^2}$ & Rate \\
\midrule
\multicolumn{4}{c}{\textit{Set~(i): $Y=0.5$, $L^2$ ceiling $p=3$}} \\[2pt]
 128 & $7.81\times10^{-2}$ & $4.63\times10^{-9}$  & ---  \\
 256 & $3.91\times10^{-2}$ & $4.09\times10^{-10}$ & 3.50 \\
 512 & $1.95\times10^{-2}$ & $3.62\times10^{-11}$ & 3.50 \\
1024 & $9.77\times10^{-3}$ & $4.31\times10^{-12}$ & 3.07 \\
\midrule
\multicolumn{4}{c}{\textit{Set~(ii): $Y=1.5$, $L^2$ ceiling $p=3$}} \\[2pt]
 128 & $7.81\times10^{-2}$ & $4.02\times10^{-6}$  & ---  \\
 256 & $3.91\times10^{-2}$ & $3.38\times10^{-7}$  & 3.57 \\
 512 & $1.95\times10^{-2}$ & $3.92\times10^{-8}$  & 3.11 \\
1024 & $9.77\times10^{-3}$ & $4.98\times10^{-9}$  & 2.98 \\
\bottomrule
\end{tabular}
\end{table}

On the coarser grids the measured $L^2$ rate is pre-asymptotic and briefly
exceeds the asymptotic prediction; as the grid is refined it settles toward the
theoretical $L^2$ ceiling $p=3$, consistent with the duality argument. The
transient overshoot on coarse grids reflects the smooth manufactured solution
not yet resolving the fractional operator at those scales, and it is
\emph{not} claimed as a theoretical convergence order: the only claim is
consistency of the measured $L^2$ error with the expected asymptotic order $p$.
The errors decrease monotonically for this benchmark and support the intended
manufactured-solution verification.

\subsection{European Put Benchmarks}
\label{ssec:european_benchmarks}

Put prices are computed by the B-spline Galerkin solver and compared against
a high-accuracy COS reference (Fang \& Oosterlee 2008) with $N_{\cos}=2048$
terms. Errors are measured on a fixed evaluation grid of 80 log-moneyness
points spanning $S_0\in[0.80K,\,1.20K]$ to avoid boundary artefacts.

\begin{table}[ht]
\centering
\caption{European put convergence, $p=3$, $N_\tau=4N$, on the log-price domain
         $[\log K-5,\log K+5]$ so that $h=10/N$.
         COS ATM reference for Set~(i): $2.9954$, Set~(ii): $6.8276$.}
\label{tab:put_convergence}
\begin{tabular}{rrrrrr}
\toprule
$N$ & $h$ & $\|e\|_{L^2}$ & Order & $\|e\|_{L^\infty}$ & CPU (s) \\
\midrule
\multicolumn{6}{c}{\textit{Set~(i): $C=0.5$, $G=M=5$, $Y=0.5$}} \\[2pt]
 64  & $1.5625\times10^{-1}$ & $1.871\times10^{-1}$ & ---  & $4.806\times10^{-1}$ & 1.7 \\
128  & $7.8125\times10^{-2}$ & $2.956\times10^{-2}$ & 2.66 & $9.026\times10^{-2}$ & 3.4 \\
256  & $3.9063\times10^{-2}$ & $3.068\times10^{-3}$ & 3.27 & $1.254\times10^{-2}$ & 7.1 \\
512  & $1.9531\times10^{-2}$ & $2.956\times10^{-4}$ & 3.38 & $1.485\times10^{-3}$ & 15.0 \\
\midrule
\multicolumn{6}{c}{\textit{Set~(ii): $C=0.1$, $G=2$, $M=3.5$, $Y=1.5$}} \\[2pt]
 64  & $1.5625\times10^{-1}$ & $1.384\times10^{-2}$ & ---  & $2.909\times10^{-2}$ & 1.7 \\
128  & $7.8125\times10^{-2}$ & $1.252\times10^{-3}$ & 3.47 & $3.704\times10^{-3}$ & 3.4 \\
256  & $3.9063\times10^{-2}$ & $1.579\times10^{-4}$ & 2.99 & $5.171\times10^{-4}$ & 7.0 \\
512  & $1.9531\times10^{-2}$ & $3.903\times10^{-5}$ & 2.02 & $1.110\times10^{-4}$ & 15.0 \\
\bottomrule
\end{tabular}
\end{table}
For Set~(ii) the finest-grid observed rate is $2.02$.  With the coupling
$N_\tau=4N$ used in this table, the Crank--Nicolson contribution scales like
$(\Delta\tau)^2\propto N^{-2}$, which decays \emph{more slowly} than the
predicted spatial contribution $h^{p-Y/2}\propto N^{-2.25}$; the finest-grid
order is therefore consistent with the onset of temporal saturation rather
than with a loss of spatial accuracy, and the clean spatial rate is the one
verified with the temporally over-resolved manufactured-solution runs of
Table~\ref{tab:manufactured}.
For Set~(i) the wider $[\log K-5,\log K+5]$ domain removes the finest-grid
order collapse seen in the earlier, narrower-domain run: the observed order
remains high through $N=512$ (about $3.38$), with no evidence of a reference-floor
artifact at the tested resolutions.  We note that Set~(i) has $Y=0.5$ and
therefore lies outside the regime covered by the convergence theory of
Section~\ref{sec:analysis}; it is reported as numerical evidence only.

\subsection{Method Comparison}
\label{ssec:method_comparison}

We compare the B-spline Galerkin scheme with alternative discretizations at the
ATM point for Set~(ii). We emphasise at the outset that the value of the
proposed method is \emph{not} raw speed for plain European vanilla pricing:
when a characteristic function is available, the COS method is essentially
unbeatable for that task, and a stabilized first-order finite-difference scheme
is fast at moderate accuracy. The B-spline advantage is its high-order
variational structure, corrected tempered connection coefficients, exact
Toeplitz interior matrices, exact spline-based Greek recovery, and compatibility
with the American obstacle formulation. Accordingly, we report both a fixed-grid
snapshot and a tolerance-based work--precision comparison.

Table~\ref{tab:numerical_method_comparison} records a fixed-grid snapshot.  The
old raw, untempered, uncompensated explicit Gr\"unwald--Letnikov baseline is too
naive for the infinite-variation CGMY generator.  In the regenerated comparison
we instead distinguish the time discretisation from the jump-operator
construction: the explicit GL row uses the same tempered/compensated operator as
the stabilized GL row but advances it by forward Euler.  At the reference grid
$N=128$, $N_\tau=512$ this explicit scheme is stable, although it remains only
conditionally stable and does blow up when $\Delta\tau\lambda_{\max}$ is made
large.  Thus the appropriate conclusion is not that all GL finite differences
are unstable for $Y>1$, but that uncompensated or too-large-time-step explicit
GL discretisations are unreliable, while a tempered/compensated Crank--Nicolson
GL baseline is a useful low-order comparator.

\begin{table}[ht]
\centering
\caption{Fixed-grid ATM put pricing error and CPU time, Set~(ii) $Y=1.5$,
$N=128$, $N_\tau=512$. COS reference price: $6.8276$. The comparison is
qualitative; a tolerance-matched comparison is given in
Table~\ref{tab:workprecision}.}
\label{tab:numerical_method_comparison}
\small
\begin{tabular}{p{4.4cm}p{2.05cm}p{1.45cm}p{3.4cm}}
\toprule
Method & $|e_{\mathrm{ATM}}|$ & CPU (s) & Notes \\
\midrule
COS (reference)                       & ---                  & $<0.001$ & CF available; European only \\
B-spline Galerkin (S1, $p=3$)         & $1.246\times10^{-3}$  & $3.358$  & high order; smooth Greeks \\
GL FD, stabilized implicit            & $1.121\times10^{-1}$  & $0.014$  & compensated; Crank--Nicolson \\
GL FD, naive explicit                 & $1.099\times10^{-1}$  & $0.012$  & stable here; conditionally stable \\
PIDE, uncompensated                   & $3.637$               & $0.015$  & omits near-origin term \\
\bottomrule
\end{tabular}
\end{table}

\paragraph{Tolerance-based (work--precision) comparison.}
The meaningful comparison fixes a target ATM accuracy and reports the CPU time
each method needs to reach it. Table~\ref{tab:workprecision} does this for the
price and for the Gamma.  At a loose price tolerance the stabilized GL baseline
is faster, while at the tighter price tolerance the B-spline solver reaches the
target and GL does not within the tested grid range.  For Gamma, however, this
particular sweep does \emph{not} show a B-spline timing crossover: the GL
central-difference estimate reaches both tested Gamma tolerances faster.  We
therefore restrict the work--precision claim to the price target and keep the
Greek advantage qualitative: the B-spline representation gives Greeks by exact
differentiation of the spline expansion, which is valuable for smooth risk
surfaces even when a finite-difference GL estimate is cheaper at the two
reported tolerances.

\begin{table}[ht]
\centering
\caption{Work--precision comparison, Set~(ii) $Y=1.5$: CPU time (s) to reach a
target ATM accuracy, for the price and for the Gamma, using the B-spline
Galerkin scheme and the stabilized implicit Gr\"unwald--Letnikov baseline. Times
are plain wall-clock measurements. A dash denotes that the target was not
reached within the tested grid range.}
\label{tab:workprecision}
\small
\begin{tabular}{lrrrrl}
\toprule
 & \multicolumn{2}{c}{Price target $|e_P|$} & \multicolumn{2}{c}{Gamma target $|e_\Gamma|$} & \\
\cmidrule(lr){2-3}\cmidrule(lr){4-5}
Method & $10^{-2}$ & $10^{-4}$ & $10^{-2}$ & $10^{-3}$ & Solver route \\
\midrule
B-spline Galerkin (S1, $p=3$)  & $2.505$ & $10.520$ & $0.624$ & $2.505$ & sparse LU; Greeks by \\
                               &         &          &         &         & analytic spline differentiation \\
Stabilized implicit GL         & $0.373$ & ---      & $0.001$ & $0.003$ & compensated GL weights \\
\bottomrule
\end{tabular}
\end{table}

This comparison gives a balanced message.  GL is the faster moderate-accuracy
baseline and is especially cheap for the two reported Gamma tolerances.  The
B-spline solver reaches the tighter price target not reached by GL in this
sweep, and its exact differentiable spline representation remains useful when
smooth Greek surfaces or compatibility with the American LCP are the primary
requirements.

\subsection{Implied-Volatility Surface and Greeks}
\label{ssec:iv_greeks}
A key advantage of the FPDE-based approach over transform methods is that a
single backward solve from $\tau=0$ to $\tau=T$ prices the European put
simultaneously at all log-strike values $x_k=x_{\min}+kh$.
Figure~\ref{fig:iv_surface} shows the implied-volatility surface produced by
this single-solve strike sweep for Set~(ii) with $N=256$, $p=3$.  Two
features are characteristic of CGMY with $Y=1.5$: the ATM implied volatility
rises only slowly with maturity, from $42.9\%$ at $T=0.10$ to $45.7\%$ at
$T=1$, reflecting the balance between the infinite-activity small jumps and
the exponential tempering; and the smile at $T=0.2$ is nearly flat in the
near-ATM region $K\in[90,110]$ with a gentle positive skew on both wings,
consistent with the mild tail asymmetry $G=2$, $M=3.5$.

\begin{figure}[ht]
\centering
\includegraphics[width=0.72\textwidth]{ 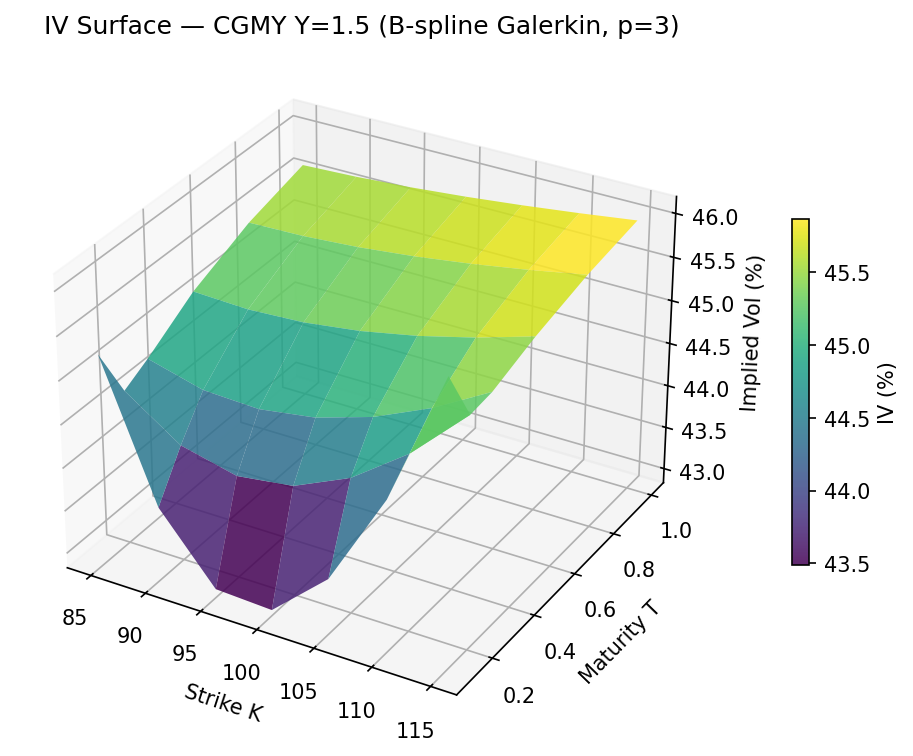}
\caption{Implied-volatility surface, Set~(ii) $Y=1.5$, $K=100$.
         B-spline Galerkin, $N=256$, $p=3$, single-solve strike sweep.}
\label{fig:iv_surface}
\end{figure}

Table~\ref{tab:greeks} reports ATM Delta, Gamma, and put prices obtained by
analytic differentiation of the spline expansion
(Section~\ref{ssec:greeks}).  Delta moves from $-0.437$ at short maturity
towards zero as $T$ grows, while Gamma decreases monotonically, consistent
with the diffusion of probability mass away from the strike under the CGMY
dynamics; the smoothness of these curves across the strike strip is the
practical payoff of the exact spline representation.

\begin{table}[ht]
\centering
\caption{Delta, Gamma, and put price at ATM ($S_0=K=100$) vs.\ maturity,
Set~(ii) $Y=1.5$.}
\label{tab:greeks}
\begin{tabular}{rrrrr}
\toprule
$T$ & $\Delta$ & $\Gamma$ & Put price \\
\midrule
0.10 & $-0.4370$ & $0.0324$ & $ 4.8975$ \\
0.20 & $-0.4144$ & $0.0216$ & $ 6.8277$ \\
0.30 & $-0.3973$ & $0.0170$ & $ 8.1840$ \\
0.50 & $-0.3706$ & $0.0126$ & $10.1015$ \\
0.75 & $-0.3445$ & $0.0099$ & $11.7255$ \\
1.00 & $-0.3231$ & $0.0083$ & $12.8736$ \\
\bottomrule
\end{tabular}
\end{table}

\subsection{American Put Extension}
\label{ssec:numerics_american}

We next test the American put extension described in
Section~\ref{ssec:american_lcp}. This is a numerical extension: the purpose is
to show that the same B-spline connection-coefficient discretisation operates
in the presence of an obstacle constraint, not to establish a convergence
theorem for the parabolic variational inequality. The reported American
experiment uses the infinite-variation CGMY parameter set
\[
  C=0.1,\qquad G=2,\qquad M=3.5,\qquad Y=1.5,
\]
with $S_0=K=100$, $r=0.1$, and $T=0.2$. For the reported table the American domain is
$[\log K-10/G,\log K+10/M]$, which avoids the boundary-truncation artifact that
appears for the wider symmetric European domain in this obstacle problem.  The
LCP at each time step is solved with the full nonlocal CGMY operator using a
matrix-free FFT/Toeplitz projected Richardson iteration, and the resulting
prices were checked against the existing banded projected-SOR implementation to
five--six significant figures. We do not use a moment-matched binomial tree as
a benchmark, since such a tree prices a lognormal GBM model rather than the
CGMY model. Instead we validate against two
independent references: the European component against a COS price, and the
American price against (i) a highly resolved run of the \emph{same} B-spline
LCP solver as a deterministic reference, and (ii) an independent
Longstaff--Schwartz Monte Carlo (LSM) estimate.

\paragraph{Longstaff--Schwartz benchmark.}
CGMY sample paths are generated by the FFT characteristic-function inversion of
Gao and Fu~\cite{GaoFu2011}: for the fixed exercise-date spacing $\Delta t=T/N_{\rm ex}$,
the increment distribution function is reconstructed once from the CGMY
characteristic function $e^{\Delta t\,\Psi(u)}$ on a fine grid and sampled by
inverse-CDF lookup, giving exact-in-law increments of the L\'evy process at the
exercise dates. We use $N_{\rm ex}=100$ equally spaced exercise dates and
$M=200{,}000$ paths with antithetic sampling. The continuation value is
regressed on a basis of $\{1,S,S^2,S^3,(K-S)^+\}$ evaluated only at
in-the-money paths, following the standard Longstaff--Schwartz prescription.
The reported batch-mean LSM price is $6.966795$ with between-batch Monte Carlo
standard error $0.014661$, giving the Student-$t$ $95\%$ confidence interval
$[6.936109,6.997481]$.  This interval contains the deterministic B-spline
reference $6.953068$.  A pooled single-regression fit on the same $200{,}000$
paths gives $6.948080$, below the deterministic reference as expected for the
usual low-bias interpretation of LSM.  We therefore use the Monte Carlo result
as a statistical consistency check, not as a strict one-sided lower bound.

\paragraph{Convergence and early-exercise premium.}
Table~\ref{tab:american_convergence} reports the convergence of the American
put price.  The early-exercise premium is
$\mathrm{EEP}=V^{\rm AM}(S_0,0)-V^{\rm EU}(S_0,0)$.
The European B-spline value converges to the COS reference price $6.8276$,
confirming the accuracy of the linear CGMY operator discretisation.  The
American value stabilises near $6.95$, corresponding to an early-exercise
premium of approximately $0.125$.

\begin{table}[ht]
\centering
\caption{American put convergence for the infinite-variation CGMY set:
$C=0.1$, $G=2$, $M=3.5$, $Y=1.5$, $r=0.1$, $T=0.2$, and $K=S_0=100$.
The European reference is the COS price $6.8276$. The American reference is the
resolved B-spline LCP value $6.953068$ at $N=512$. The independent LSM-CGMY
cross-check gives $6.966795$ with standard error $0.014661$ and $95\%$ interval
$[6.936109,6.997481]$, which contains the deterministic reference.}
\label{tab:american_convergence}
\begin{tabular}{rrrrrrr}
\toprule
$N$ & $N_T$ & European & American & EEP & $\max(\psi_A-V_h)^+$ & CPU (s) \\
\midrule
 64 &  256 & 6.835468 & 7.045229 & 0.209761 & $1.250\times10^{-1}$ & 3.6 \\
128 &  512 & 6.828104 & 6.971676 & 0.143572 & $8.553\times10^{-3}$ & 8.7 \\
256 & 1024 & 6.827686 & 6.955636 & 0.127950 & $3.109\times10^{-5}$ & 20.4 \\
512 & 2048 & 6.827664 & 6.953068 & 0.125405 & $2.142\times10^{-6}$ & 92.4 \\
\bottomrule
\end{tabular}
\end{table}

The American price decreases monotonically under refinement while remaining
above the European value, as required by the obstacle constraint, and
converges to the resolved deterministic reference, which lies inside the LSM
$95\%$ interval reported in the caption; the Monte Carlo experiment therefore
supports statistical consistency rather than a strict one-sided lower-bound
ordering.  The pooled single-regression value $6.948080$ sits slightly below
the deterministic reference, consistent with the usual LSM low-bias
heuristic.

Across the grid we also verified the pointwise payoff-constraint check of
Remark~\ref{rem:discrete_obstacle}: the maximum violation
$\max_x(\psi_A(x)-V_h^{N_T}(x))^+$ at the evaluation and Gauss quadrature nodes
decreased from $1.250\times10^{-1}$ at $N=64$ to $2.142\times10^{-6}$ at
$N=512$, i.e. below the finest-grid spatial-error proxy.

\paragraph{Early-exercise boundary.}
The contact set is recovered from the active obstacle constraint in
\eqref{eq:american_coeff_lcp}.  Table~\ref{tab:american_boundary} reports the
approximate critical stock-price interval $S_f(\tau)$ below which immediate
exercise is optimal.  The boundary decreases as the time to maturity
increases, which is the expected qualitative behaviour for an American put.

\begin{table}[ht]
\centering
\caption{Approximate early-exercise boundary $S_f(\tau)$ for the American put
under the infinite-variation CGMY set, computed with $N=128$ and $N_T=512$.}
\label{tab:american_boundary}
\begin{tabular}{cc}
\toprule
$\tau$ & $S_f(\tau)$ \\
\midrule
0.0000 & 97.5--99.0 \\
0.0286 & 88.0--90.0 \\
0.0571 & 83.0--85.0 \\
0.0857 & 80.0--82.0 \\
0.1143 & 78.0--79.5 \\
0.1429 & 76.5--78.0 \\
0.1714 & 75.0--76.5 \\
0.2000 & 73.5--75.0 \\
\bottomrule
\end{tabular}
\end{table}

\paragraph{PSOR relaxation.}
A relaxation study for the banded direct-factorisation (S3-type)
projected-SOR cross-check at $N=128$, reported in the supplementary materials
(section~SM5), shows the iteration count minimised near
$\omega_{\rm SOR}=1.1$; larger relaxation values increase the count because
the CGMY stiffness matrix is nonsymmetric and nonlocal, while the converged
price agrees to within $10^{-4}$ for all stable choices.


\section{Real-Data Calibration to S\&P 500 Option Surfaces}
\label{sec:calibration}

This section applies the B-spline Galerkin pricing engine to market data.
The goal is twofold: first, to test whether the CGMY model calibrated through
the proposed FPDE solver reproduces observed implied-volatility smiles across
different regimes; second, to record the computational cost of using the
B-spline solver inside a full nonlinear calibration loop.  Detailed data
filtering, forward construction, parameter transformations, and optimizer
settings are collected in the supplementary materials (section~SM4).

\subsection{Data Set and Surface Construction}
\label{ssec:calib_data}

We use end-of-day SPX European option chains from OptionMetrics on three
dates representing distinct regimes:
\begin{center}
\begin{tabular}{lll}
\toprule
Date & Regime & Interpretation \\
\midrule
2020-01-15 & Low-volatility & pre-COVID market, VIX approximately $13$ \\
2020-03-18 & Crisis & acute COVID-19 selloff, VIX approximately $76$ \\
2021-09-15 & Recovery & post-COVID recovery, VIX approximately $20$ \\
\bottomrule
\end{tabular}
\end{center}
For each date, we retain liquid out-of-the-money quotes near the forward:
puts for $K<F$ and calls for $K\ge F$, with $K/F\in[0.85,1.15]$.  The primary
filter removes crossed or zero quotes and keeps only quotes with relative
bid--ask spread below $5\%$.  A relaxed $15\%$ filter is used only as a
robustness and out-of-sample diagnostic for the crisis date, where market
spreads were extremely wide.

\begin{table}[htbp]
\centering
\caption{Surviving OTM quotes after all filters.  The crisis date loses the
1w, 6m, and 1y maturities under the primary $5\%$ spread filter.}
\label{tab:quote_counts}
\small
\begin{tabular}{p{1.65cm}p{2.05cm}p{2.05cm}p{2.05cm}p{2.75cm}}
\toprule
Date & Regime & Spread $<5\%$ & Spread $<15\%$ & Notes \\
\midrule
2020-01-15 & Low-volatility & 224, 5 mats. & --- & full surface \\
2020-03-18 & Crisis & 113, 2 mats. & 243, 5 mats. & wide-spread tenors excluded \\
2021-09-15 & Recovery & 575, 5 mats. & --- & largest surface \\
\bottomrule
\end{tabular}
\end{table}

For each maturity the forward $F$ and discount factor $D$ are extracted from
put--call parity rather than imposed externally.  This guarantees that the
model uses the same maturity-specific forward as the observed surface.
\paragraph{Data overview across market regimes.}
Before calibrating the CGMY model, we first inspect the raw implied-volatility
surfaces used in the real-data study.  Figure~\ref{fig:data_overview}
summarizes the three calibration dates: a low-volatility pre-COVID date
(15 January 2020), the COVID-19 crisis date (18 March 2020), and a
post-COVID recovery date (15 September 2021).  The panels show pronounced
changes in both smile level and term structure.  The calm and recovery dates
exhibit downward-sloping equity-index volatility smiles with moderate ATM
volatility levels, whereas the crisis date displays extremely elevated
short-maturity implied volatility and an inverted term structure.  This
regime variation provides a useful stress test for the B-spline CGMY
calibration procedure.
\begin{figure}[htbp]
  \centering
  \includegraphics[width=\textwidth]{ 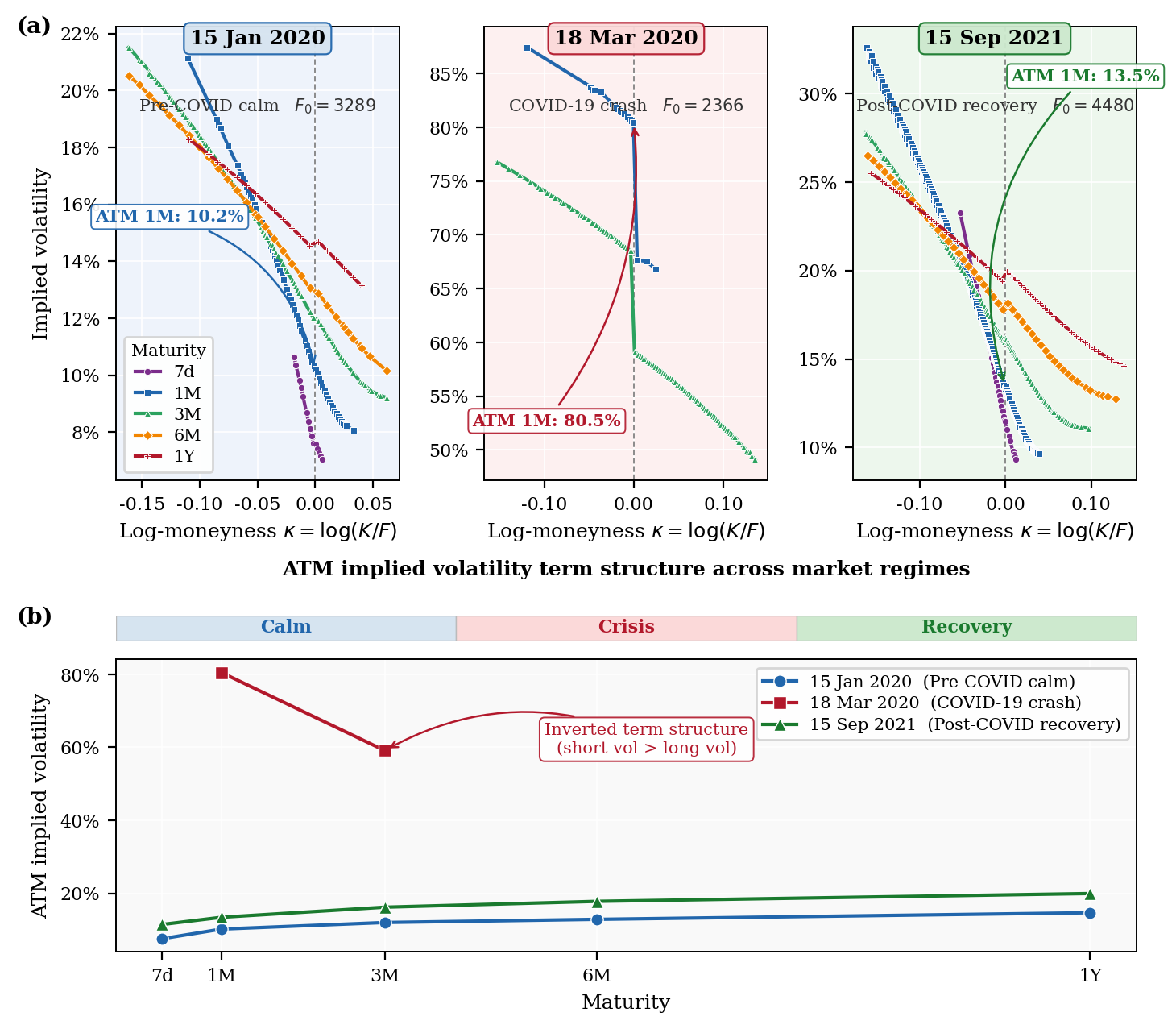}
  \caption{
  Market implied-volatility data used in the real-data calibration study.
  Panel (a) shows implied-volatility smiles against log-moneyness
  $\kappa=\log(K/F)$ for three market regimes: pre-COVID calm
  (15 January 2020), COVID-19 crisis (18 March 2020), and post-COVID
  recovery (15 September 2021).  The vertical dashed line marks ATM
  moneyness.  The crisis panel is shown after liquidity filtering and has a
  different volatility scale because short-maturity implied volatilities are
  much larger.  Panel (b) reports the corresponding ATM implied-volatility
  term structures, highlighting the strongly inverted crisis term structure.
  }
  \label{fig:data_overview}
\end{figure}
The calibration experiments below use these three surfaces to test the method
under qualitatively different regimes: a smooth low-volatility surface, a
stressed crisis surface with sparse and high-volatility short-maturity quotes,
and a recovery surface with a still-elevated but more regular smile structure.
\subsection{Calibration Objective}
\label{ssec:calib_objective}

Let $\theta=(C,G,M,Y)$ denote the CGMY parameter vector.  Calibration is
performed in price space with a vega normalization,
\begin{equation}
\label{eq:calib_loss}
  \mathcal L(\theta)
  = \frac{1}{N_q}\sum_{i=1}^{N_q}
  \left(
  \frac{V_{\mathrm{model},i}(\theta)-V_{\mathrm{mid},i}}
       {\max(\nu_i,\nu_{\mathrm{floor}})}
  \right)^2,
\end{equation}
where $V_{\mathrm{mid},i}$ is the market mid-price, $\nu_i$ is the
Black--Scholes forward vega, and $\nu_{\mathrm{floor}}=10^{-3}$.  The vega
normalization makes the loss approximately an implied-volatility RMSE while
still avoiding repeated IV inversions inside the optimizer.

The parameter bounds are
\[
  C\in[10^{-3},20],\qquad
  G\in[0.1,80],\qquad
  M\in[1.01,80],\qquad
  Y\in[1.05,1.95].
\]
The lower bound $M>1$ enforces the exponential moment needed for the stock
price, while $Y\ge1.05$ keeps the calibration inside the infinite-variation
regime analysed in Section~\ref{sec:analysis}.  We use a two-stage optimizer:
Differential Evolution at the coarse grid level to locate the basin, followed
by L-BFGS-B at the refined grid level.  In each objective evaluation, one
backward solve prices all strikes of a fixed maturity by the homogeneity
identity described in the supplementary materials (section~SM4).

\subsection{Calibrated Parameters and Fit Quality}
\label{ssec:calib_results}

Table~\ref{tab:calib_main_results} reports the calibrated parameters and
in-sample implied-volatility errors.  Errors are reported in vol-points,
i.e. decimal implied volatility multiplied by $100$.

\begin{table}[htbp]
\centering
\caption{SPX calibration results.  The first, second, and fourth rows use the
primary $5\%$ spread filter.  The $15\%$ crisis row is a robustness check.}
\label{tab:calib_main_results}
\setlength{\tabcolsep}{5pt}
\begin{tabular}{llcrrrrcc}
\toprule
Date & Regime & Spread & $C^*$ & $G^*$ & $M^*$ & $Y^*$
& RMSE & Max \\
&&&&&&& (\vp) & (\vp) \\
\midrule
2020-01-15 & Low-volatility & $5\%$
  & 0.049 & 2.236 & 80.0 & 1.089 & \textbf{2.47} & 9.80 \\
2020-03-18 & Crisis & $5\%$
  & 0.543 & 0.100 & 27.8 & 1.050 & 7.19 & 17.10 \\
2020-03-18 & Crisis & $15\%$
  & 0.546 & 0.100 & 6.7 & 1.050 & 11.05 & 46.53 \\
2021-09-15 & Recovery & $5\%$
  & 0.097 & 2.594 & 80.0 & 1.050 & \textbf{1.97} & 8.09 \\
\bottomrule
\end{tabular}
\end{table}

The two well-identified five-maturity surfaces achieve in-sample IV RMSEs of
$2.47\,\vp$ and $1.97\,\vp$.  The crisis date has a larger error
($7.19\,\vp$) because the primary spread filter leaves only two maturities,
30d and 93d.  This should be interpreted as a data-quality and
identifiability limitation rather than a failure of the deterministic pricing
engine.

Table~\ref{tab:calib_per_maturity} gives the maturity-level decomposition.
The ultra-short 1w smile is the hardest part of the surface on the two
well-identified dates, while the 3m, 6m, and 1y smiles are fitted more
accurately.  For the crisis date, starred entries are not used in the
$5\%$-spread fit; they use $15\%$-spread quotes only to evaluate the
extrapolation of the two-maturity calibration.

\begin{table}[htbp]
\centering
\caption{Per-maturity IV RMSE in vol-points.  The notation ``RMSE/quotes''
indicates the error and the number of quotes in that maturity bucket.
Starred crisis entries are $15\%$-spread fallback quotes evaluated
out-of-sample under the $5\%$-spread calibration.}
\label{tab:calib_per_maturity}
\setlength{\tabcolsep}{5pt}
\begin{tabular}{lcccc}
\toprule
Tenor & $T$ (years) & 2020-01-15 & 2020-03-18 & 2021-09-15 \\
\midrule
1w  & 0.019 & 8.52 / 15 & 40.40$^*$ / 9 & 5.38 / 34 \\
1m  & 0.082 & 1.63 / 61 & 14.10 / 19 & 1.46 / 266 \\
3m  & 0.255 & 0.56 / 96 & 4.68 / 94 & 1.07 / 179 \\
6m  & 0.46--0.50 & 0.88 / 32 & 10.16$^*$ / 19 & 2.05 / 49 \\
1y  & 1.003 & 1.76 / 19 & 10.44$^*$ / 7 & 2.44 / 42 \\
\midrule
All fit quotes & --- & \textbf{2.47 / 224} & \textbf{7.19 / 113}
& \textbf{1.97 / 575} \\
All including fallback & --- & \textbf{2.47 / 224} & \textbf{12.53 / 148}
& \textbf{1.97 / 575} \\
\bottomrule
\end{tabular}
\end{table}

\subsection{Real-Data Implied-Volatility Smiles and Residuals}
\label{ssec:calib_figures}

Figure~\ref{fig:calib_smiles} shows market and model implied-volatility
smiles.  The B-spline CGMY model captures the shape of the low-volatility
and recovery surfaces well across the liquid maturities.  The crisis date is
visibly less stable: the 30d smile is systematically underestimated near the
money, and the 1w, 6m, and 1y panels are marked as fallback data because they
are not present after the primary $5\%$ spread filter.

\begin{figure}[htbp]
  \centering
  \includegraphics[width=\textwidth]{ 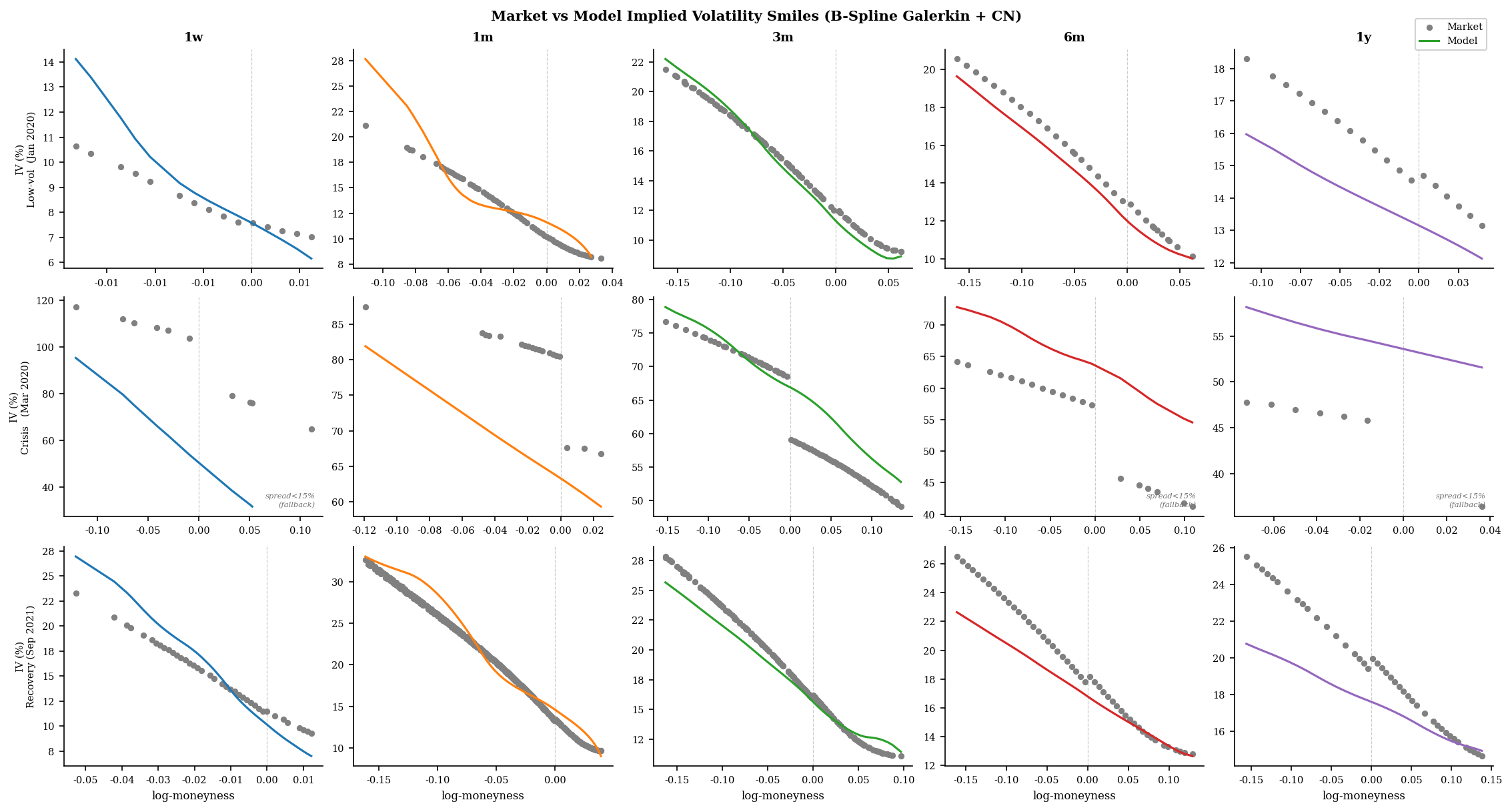}
  \caption{Market versus model implied-volatility smiles for the three SPX
  dates and five nominal tenors.  Grey markers are market quotes and coloured
  curves are B-spline Galerkin--CN CGMY prices converted to implied
  volatility.  For the crisis row, 1w, 6m, and 1y panels use $15\%$-spread
  fallback quotes and are therefore out-of-sample checks of the $5\%$ fit.}
  \label{fig:calib_smiles}
\end{figure}

The residual plot in Figure~\ref{fig:calib_residuals} gives the same
information quote by quote.  The low-volatility and recovery residuals are
mostly centred near zero, with larger deviations concentrated at very short
maturities.  The crisis residuals reveal the main model-data conflict: the
30d smile is too low by roughly 10--17 vol-points near the money, while the
out-of-sample 1w fallback residuals are much larger.  This confirms that the
crisis calibration should be read as a stress-test of the pipeline, not as a
well-identified five-maturity fit.

\begin{figure}[htbp]
  \centering
  \includegraphics[width=\textwidth]{ 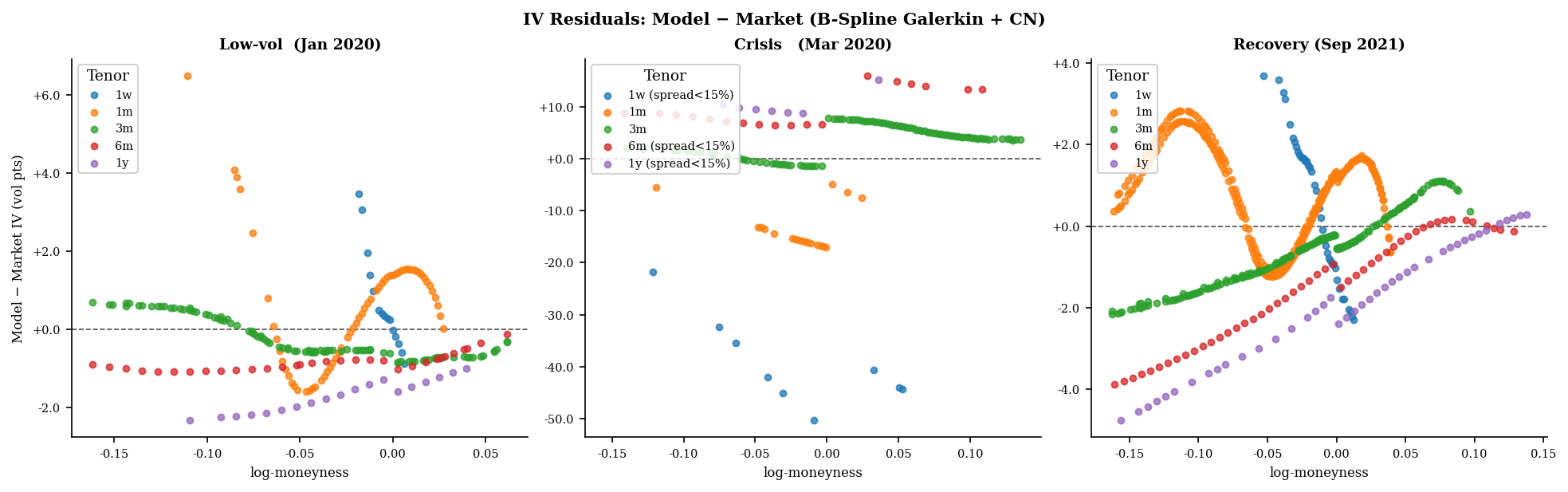}
  \caption{IV residuals, model minus market, in vol-points.  Each point is one
  OTM quote and colour denotes tenor.  The dashed line is zero residual.  The
  crisis 1w, 6m, and 1y residuals are fallback/out-of-sample diagnostics.}
  \label{fig:calib_residuals}
\end{figure}

\subsection{Financial Interpretation and Crisis Caveat}
\label{ssec:calib_interpretation}

The calibrated parameters have a consistent financial interpretation.  The
jump-intensity parameter rises sharply in the crisis regime:
\[
  C^*_{\mathrm{low}}=0.049,\qquad
  C^*_{\mathrm{crisis}}=0.543,\qquad
  C^*_{\mathrm{recovery}}=0.097.
\]
Thus the crisis value is approximately $11$ times the low-volatility value
and $5.6$ times the recovery value.  Since $C$ scales the L\'{e}vy density,
this indicates a much larger risk-neutral intensity of large moves during
March 2020.

The fine-structure index is comparatively stable:
\[
  Y^*_{\mathrm{low}}=1.089,
  \qquad Y^*_{\mathrm{crisis}}=1.050,
  \qquad Y^*_{\mathrm{recovery}}=1.050.
\]
The range is only $0.039$, about $3.6\%$ of the cross-date mean.  However,
two estimates sit at the lower bound $Y=1.05$, so they should be interpreted
as constrained estimates indicating a preference for a value near the lower
edge of the infinite-variation regime.

The tail parameters are less uniformly identified.  On the non-crisis dates,
$M^*=80$ hits the upper bound, which is consistent with limited information
about the far right tail inside the near-ATM moneyness window.  On the crisis
date, $G^*=0.1$ hits the lower bound, indicating an extremely heavy left tail
under the risk-neutral measure.  The crisis robustness check confirms that
$C^*$, $G^*$, and $Y^*$ are stable when the spread filter is relaxed from
$5\%$ to $15\%$, while $M^*$ changes substantially because longer-maturity
OTM calls re-enter the surface.

\begin{table}[htbp]
\centering
\caption{Crisis-date robustness: primary $5\%$ spread filter versus relaxed
$15\%$ filter.}
\label{tab:calib_crisis_robustness}
\begin{tabular}{lrrrr}
\toprule
Parameter & Spread $5\%$ & Spread $15\%$ & $|\Delta|$ & Relative change \\
\midrule
$C^*$ & 0.5431 & 0.5462 & 0.003 & 0.6\% \\
$G^*$ & 0.100 & 0.100 & 0.000 & 0.0\% \\
$M^*$ & 27.79 & 6.65 & 21.14 & 76.1\% \\
$Y^*$ & 1.050 & 1.050 & 0.000 & 0.0\% \\
\midrule
IV RMSE & 7.19\,\vp & 11.05\,\vp & 3.87\,\vp & 53.8\% \\
\bottomrule
\end{tabular}
\end{table}

The March 2020 result is therefore included for completeness but not treated
as the primary evidence of calibration accuracy.  The main evidence is the
low-volatility and recovery surfaces, where all five maturities survive the
primary quality filter and the CGMY model is meaningfully identified.

\subsection{Calibration Time and GL Benchmark}
\label{ssec:calib_speed}

Table~\ref{tab:calib_timing} records the wall-clock time of the B-spline
calibrations.  Five-maturity surfaces require about 4--5 minutes in the
reported Python implementation.  The crisis $5\%$ run is faster because only two maturities survive the
quality filter.

\begin{table}[htbp]
\centering
\caption{Wall-clock time per B-spline calibration run in the reported Python implementation.}
\label{tab:calib_timing}
\setlength{\tabcolsep}{6pt}
\begin{tabular}{llcrrrr}
\toprule
Date & Spread & Maturities & Quotes & DE evals & L-BFGS-B evals & Total \\
\midrule
2020-01-15 & $5\%$ & 5 & 224 & 9{,}420 & 235 & 255 s \\
2020-03-18 & $5\%$ & 2 & 113 & 11{,}340 & 170 & 120 s \\
2020-03-18 & $15\%$ & 5 & 243 & 10{,}500 & 130 & 275 s \\
2021-09-15 & $5\%$ & 5 & 575 & 9{,}360 & 265 & 261 s \\
\bottomrule
\end{tabular}
\end{table}

For context, we also run an accuracy-matched comparison with a
Gr\"unwald--Letnikov finite-difference discretisation of the same backward
CGMY FPDE on the crisis surface.  The optimizer settings are identical for
both methods.  The comparison is intentionally accuracy-matched at the
moderate tolerance needed to identify the calibration basin, not at the
asymptotic high-accuracy regime.

\begin{table}[htbp]
\centering
\caption{Accuracy-matched B-spline versus GL calibration benchmark on the
2020-03-18 crisis surface.  IV RMSE is measured in vol-points.}
\label{tab:calib_gl_benchmark}
\setlength{\tabcolsep}{5pt}
\begin{tabular}{lcccrrrr}
\toprule
Method & $N_{\mathrm{DE}}$ & $N_{\mathrm{LBFGS}}$
& $C^*$ & $G^*$ & $Y^*$ & IV RMSE & Time \\
\midrule
B-spline Galerkin & 32 & 64 & 0.543 & 0.100 & 1.050 & 7.19\,\vp & 115 s \\
GL finite difference & 64 & 128 & 0.582 & 0.100 & 1.050 & 6.66\,\vp & 20 s \\
\bottomrule
\end{tabular}
\end{table}

The GL solver is faster in this moderate-accuracy calibration experiment:
\[
  \alpha = \frac{T_{\mathrm{GL}}}{T_{\mathrm{BS}}}
  = \frac{20}{115}=0.17,
\]
so GL is about $5.8$ times faster.  The result should not be read as a
contradiction of the B-spline accuracy theory.  In this experiment, the
calibration target is coarse enough that first-order GL with twice as many
grid points reaches the correct basin.  For stricter pricing or risk targets,
the B-spline rate $O(h^{p-Y/2})$ becomes the decisive advantage over the
first-order GL discretisation.

\begin{figure}[htbp]
  \centering
  \includegraphics[width=0.58\textwidth]{ 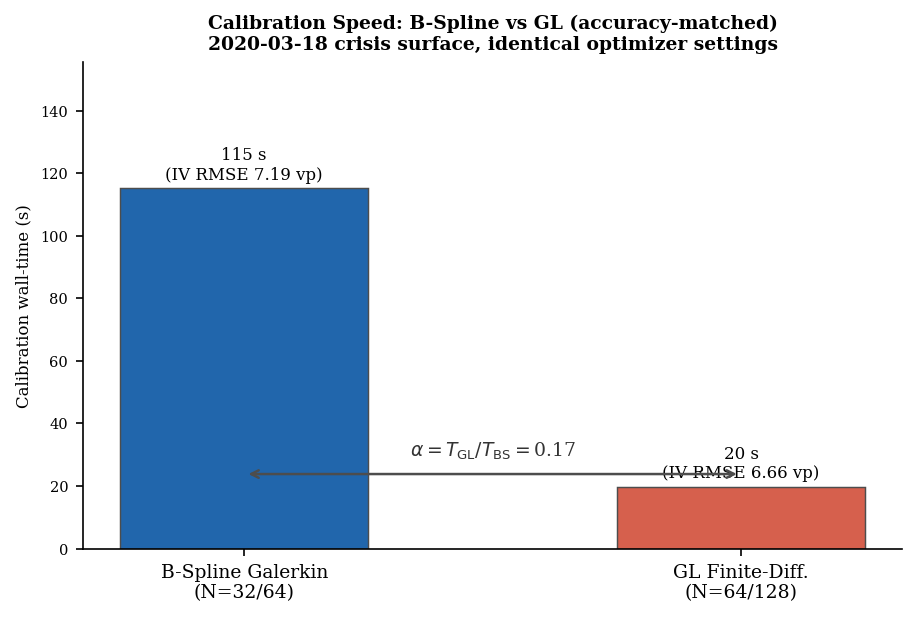}
  \caption{Accuracy-matched calibration wall-time on the 2020-03-18 crisis
  surface.  GL is faster in this moderate-accuracy calibration regime, while
  the B-spline scheme retains the higher-order spatial convergence advantage
  for stricter accuracy requirements.}
  \label{fig:calib_speed}
\end{figure}

\paragraph{Calibration takeaway.}
The real-data experiment shows that the proposed B-spline Galerkin solver can
be used as the deterministic pricing engine inside a complete CGMY calibration
pipeline.  It gives accurate fits on well-identified SPX surfaces, exposes the
identifiability limitations of the crisis surface, and provides smooth model
smiles and residual diagnostics across maturities.  The GL comparison gives a
balanced computational conclusion: GL is a useful fast baseline at moderate
accuracy, whereas the B-spline solver is the natural choice when high-order
spatial accuracy, smoother Greeks, or stricter pricing tolerances are required.

\section{Conclusion and Extensions}
\label{sec:conclusion}

We developed and analyzed a B-spline scaling-function Galerkin method for the
backward CGMY FPDE.  The corrected Fourier-symbol construction of the
tempered connection-coefficient matrices yields exact Toeplitz interior
blocks with the correct fractional scaling $2^{jY}$, $O(N\log N)$ FFT-based
application, and closed-form untempered limits as implementation checks.  In
the infinite-variation regime $Y\in(1,2)$ we proved full-line and
finite-domain G{\aa}rding inequalities --- the latter by a zero-extension
reduction that leaves no unverified coercivity assumption --- showed both to
be coercive whenever the risk-free rate is positive, and on this foundation
established shifted elliptic-projection, semidiscrete, and fully discrete
error estimates with spatial rate $O(h^{p-Y/2})$, together with a forced
shifted-energy Crank--Nicolson stability estimate that, since $\sigma=0$ is
admissible for $r>0$, governs the plain scheme exactly as implemented.  The
numerical experiments confirm the predicted rates, extend the same
discretisation to American puts through an $L^2$-projected obstacle and a
projected-SOR linear complementarity solve --- validated against a resolved
deterministic reference and a Longstaff--Schwartz Monte Carlo cross-check ---
and demonstrate real-data calibration on S\&P~500 surfaces across
low-volatility, crisis, and recovery regimes.  The work--precision and
crisis-date calibration benchmarks against a stabilized
Gr\"unwald--Letnikov baseline give a balanced computational conclusion: the
first-order baseline is faster at moderate accuracy, while the B-spline
solver is decisive when tight pricing tolerances, smooth spline-based Greeks,
or compatibility with the American obstacle problem are required.

Several extensions are natural.  A convergence theory for the parabolic
obstacle problem would put the American solver on the same rigorous footing
as the European scheme, and the finite-variation regime $Y\in(0,1)$ ---
covered by the connection-coefficient construction and the experiments but
not by the variational proofs --- calls for a separate drift-dominated
argument.  Beyond these, tensor-product B-spline spaces open multi-asset CGMY
pricing for basket options; the corresponding forward FPDE would deliver
transition densities for likelihood-based estimation directly in the scaling
basis; and the same connection-coefficient machinery adapts to other tempered
stable models such as VG and NIG, ideally combined with adaptive
multiresolution refinement near the payoff kink at $x=\log K$.

\section*{Acknowledgements}

Wharton Research Data Services (WRDS) was used in preparing this work. This service and the data available thereon constitute valuable intellectual property and trade secrets of WRDS and/or its third-party suppliers.


\appendix

\section{Fractional Operators: Notation and Key Identities}
\label{app:operators}

\begin{table}[ht]
\centering
\caption{Summary of fractional operators used in this paper.
Here $m = \lceil Y \rceil$ (the ceiling of $Y$).}
\label{tab:operators}
\small
\begin{tabular}{p{1.8cm} p{2.55cm} p{4.75cm} p{2.65cm}}
\toprule
Symbol & Name & Definition & Role \\
\midrule
$\RLleft{Y}{x} f$ &
Left-sided RL &
$\frac{1}{\Gamma(m-Y)}\frac{d^m}{dx^m}\int_{-\infty}^x \frac{f(y)}{(x-y)^{Y+1-m}}dy$ &
Backward FPDE (negative jumps) \\[8pt]
$\RLright{Y}{x} f$ &
Right-sided RL &
$\frac{(-1)^m}{\Gamma(m-Y)}\frac{d^m}{dx^m}\int_x^\infty \frac{f(y)}{(y-x)^{Y+1-m}}dy$ &
Backward FPDE (positive jumps) \\[8pt]
$\Cleft{Y}{x} f$ &
Left-sided Caputo &
$\frac{1}{\Gamma(m-Y)}\int_{-\infty}^x \frac{f^{(m)}(y)}{(x-y)^{Y+1-m}}dy$ &
Forward Kolmogorov FPDE \\[8pt]
${}^C_{x}D^Y_\infty f$ &
Right-sided Caputo &
$\frac{(-1)^m}{\Gamma(m-Y)}\int_x^\infty \frac{f^{(m)}(y)}{(y-x)^{Y+1-m}}dy$ &
Forward Kolmogorov FPDE \\
\bottomrule
\end{tabular}
\end{table}

Key identity (RL--Caputo relationship,
see \cite{Fallahgoul2017}, Proposition 2):
\begin{equation}
\Cleft{Y}{x} f(x)
= \RLleft{Y}{x} f(x)
- \sum_{k=0}^{m-1}\frac{f^{(k)}(-\infty)}{\Gamma(k-Y+1)}(x-(-\infty))^{k-Y}.
\end{equation}
When $f$ and its first $m-1$ derivatives vanish at $-\infty$
(as for the CGMY pricing solution), the two operators coincide.

Fourier representation:
\begin{equation}
\widehat{\RLleft{Y}{x} f}(\xi) = (i\xi)^Y\,\hat{f}(\xi),
\qquad
\widehat{\RLright{Y}{x} f}(\xi) = (-i\xi)^Y\,\hat{f}(\xi),
\end{equation}
where $(i\xi)^Y = |\xi|^Y\exp\!\left(i\frac{\pi Y}{2}\mathrm{sgn}(\xi)\right)$.

\bibliographystyle{siamplain}
\bibliography{biblib}
\end{document}